\documentclass[a4paper,reqno,12pt]{amsart}
\usepackage[T1]{fontenc}
\usepackage[utf8]{inputenc}
\usepackage{amsmath,amssymb,amsfonts,amsthm,amscd}
\usepackage{mathrsfs}
\usepackage{bm}
\usepackage{bbm}

\usepackage{tikz}
\usetikzlibrary{arrows.meta, positioning, calc}
\usepackage{pgf}
\usepackage{tikz-cd}
\usepackage{tkz-graph}
\usepackage[all,2cell]{xy}

\usepackage{graphicx}
\usepackage{multicol}
\usepackage{enumerate}
\usepackage{comment}
\usepackage{rotating}
\usepackage[normalem]{ulem}
\usepackage{xcolor}

\usetikzlibrary{decorations.markings,arrows}

\tikzset{
  midarrow/.style={
    postaction=decorate,
    decoration={
      markings,
      mark=at position .5 with {\arrow{stealth}}
    }
  }
}

\allowdisplaybreaks

\UseAllTwocells
\SilentMatrices

\numberwithin{equation}{section}

\usepackage[
  pagebackref,
  colorlinks=true,
  linkcolor=blue,
  citecolor=purple,
  urlcolor=blue,
  hypertexnames=false
]{hyperref}

\newtheorem{thm}{Theorem}[subsection]
\newtheorem{cor}[thm]{Corollary}
\newtheorem{lem}[thm]{Lemma}
\newtheorem{prop}[thm]{Proposition}

\theoremstyle{definition}
\newtheorem{dfn}[thm]{Definition}

\theoremstyle{remark}
\newtheorem{rmk}[thm]{Remark}

\newtheorem{example}[thm]{Example}

\def\FF{\mathsf{F}}

\newcommand{\image}{\operatorname{Im}}

\newcommand{\Aut}{\operatorname{Aut}}
\newcommand{\Ext}{\operatorname{Ext}}
\newcommand{\MCE}{\operatorname{MCE}}

\newcommand{\lspan}{\operatorname{span}}
\newcommand{\sur}{\operatorname{Sur}}

\newcommand{\gr}{\operatorname{gr}}

\def\-{\text{-}}

\newcommand{\KP}{\operatorname{KP}}

\newcommand{\N}{\mathbf{n}}
\newcommand{\M}{\mathbf{m}}
\newcommand{\bP}{\mathbf{p}}
\newcommand{\Q}{\mathbf{q}}
\newcommand{\bT}{\mathbf{t}}
\newcommand{\mS}{\mathcal{S}}

\begin{document}

\pagestyle{headings}

\title[Talented monoids of higher-rank graphs]{Some more talents of the talented monoid of\\ a higher-rank graph}

\author{Roozbeh Hazrat}
\address{Roozbeh Hazrat: Centre for Research in Mathematics and Data Science\\Western Sydney University, Australia} \email{r.hazrat@westernsydney.edu.au}

\author{Huanhuan Li}
\address{Huanhuan Li: Center for Pure Mathematics, School of Mathematical Sciences, Anhui University, Hefei, 230601, China} \email{lhh@ahu.edu.cn}

\author{Promit Mukherjee}
\address{Promit Mukherjee: Department of Mathematics, Jadavpur University, Kolkata-700032, India} \email{promitmukherjeejumath@gmail.com}



\begin{abstract}
In this paper, we explore the idea that the graded Grothendieck group $K_0^{\gr}$, or equivalently its positive cone, the talented monoid, can detect the \emph{structural type} of higher-rank graph algebras (i.e., higher-rank graph $C^*$-algebras and Kumjian--Pask algebras). 

We show that the talented monoid captures some of the essential geometric information of a higher-rank graph, including the existence of cycles with and without entrances. In turn, we show that the graded $K$-theory can effectively distinguish the class of locally finite Kumjian--Pask algebras, and also the class of crossed product Kumjian--Pask algebras. We also derive talented monoid criteria for higher-rank graph algebras to be purely infinite simple, and not to be $AF$ or ultramatricial. 
\end{abstract}

\maketitle

\section{Introduction}\label{sec intro}
Graph algebras provide a remarkably effective setting in which algebraic structure and combinatorial geometry can be studied simultaneously. In the case of Leavitt path algebras, much of the structure of the algebra is encoded in the underlying directed graph, and this has led to a rich interplay between graph theory, graded ring theory, algebraic $K$-theory, and symbolic dynamics~\cite{lpabook,AbramsHazrat}. A fundamental problem in this area is to determine which invariants retain enough information to recover the structural type of the algebra. One of the most compelling candidates is the graded Grothendieck group $K_0^{\mathrm{gr}}$~\cite{haz2013}. For a $\mathbb Z$-graded Leavitt path algebra $L_\FF(E)$ associated to the directed graph $E$ with coefficients in a field $\FF$, it has now been established that the group
$K_0^{\mathrm{gr}}(L_\FF(E))$ carries considerably more information than its ungraded counterpart $K_0(L_\FF(E))$. The positive cone of $K_0^{\mathrm{gr}}(L_\FF(E))$ is naturally a monoid with a $\mathbb Z$-action, which can be described directly from the underlying graph, commonly called the \emph{talented monoid} $T_E$ of the graph $E$~\cite{hazli,Luiz}. \emph{The graded classification conjecture} suggests that this invariant is  sufficiently rich to classify Leavitt path algebras completely~\cite{haz2013}. In particular, the graded classification conjecture would have consequences extending beyond the classification of graph algebras, including connections with the classification of shifts of finite type in symbolic dynamics. We refer the reader to \cite{AbramsRuizTomforde,AraPardo,AraDoNam,ArnoneLifting,arnonecor, ArnoneCortinasNonexistence,BockHazratSebandal,willie,DoHazratNam,EilersRuizSims,VasFull,VasGradedClassification} for work around the graded classification conjecture in the setting of Leavitt path algebras. 

In this paper, we extend the study of the graded Grothendieck group and its positive cone to the setting of algebras associated to higher-rank graphs~\cite{KumjianPask,Pino}. Higher-rank graphs, or $k$-graphs, generalize directed graphs by replacing the usual path-length function with a degree functor $d:\Lambda\longrightarrow \mathbb N^k$. The associated higher-rank $C^*$-algebras $C^*(\Lambda)$ and their discrete version, the so called, Kumjian--Pask algebras $\KP_\FF(\Lambda)$ are naturally $\mathbb Z^k$-graded, and their graded Grothendieck groups therefore carry a $\mathbb Z^k$-action. The higher dimension makes the relationship between the graph and its algebraic invariants substantially more intricate than in the one-dimensional case as we will see in this paper. The study of the graded Grothendieck group of Kumjian--Pask algebras and their positive cones, i.e., the talented monoid of higher-rank graphs were initiated in \cite{HMPS1,HMPS2}. 

In this paper, we pursue the idea that the graded Grothendieck group $K_0^{\gr}$, or equivalently, the talented monoid, can detect the \emph{structural type} of a higher-rank graph algebra. 
The program we explore is the following: given higher-rank $k$-graphs $\Lambda$ and $\Omega$, if there is an $\mathbb Z^k$-monoid isomorphism between the talented monoids $T_\Lambda$ and $T_\Omega$, then to what extent can one conclude that the associated $C^*$-algebras $C^*(\Lambda)$ and $C^*(\Omega)$ as well as Kumjian--Pask algebras $\operatorname{KP}_\FF(\Lambda)$ and $\operatorname{KP}_\FF(\Omega)$ have the same structural type, such as simple, purely infinite simple, $AF$ (ultramatricial), locally finite, crossed product etc.? 

In this direction, we start by observing that the talented monoid  $T_\Lambda$ of the $k$-graph $\Lambda$, retains detailed information about the geometry of the graph. For example, properties concerning the existence of cycles and entrances can be expressed in terms of the order and the grading action on the monoid. In particular, the existence of a nontrivial loop without an entrance is reflected by a nonzero element of the monoid which is fixed by a nonzero element of the $\mathbb Z^k$-action, whereas a loop with an entrance produces strict inequality under a positive translate (Theorem~\ref{th characterizing loops via TM}). This observation is important for the structural question that motivates the paper. In the theory of graph algebras, loops with entrances and loops without entrances play fundamentally different roles. The former are closely related to infinite behavior and purely infinite phenomena, while the latter are associated with periodic or crossed-product behavior. One of the consequences is a structural characterization of locally finite Kumjian--Pask algebras (Theorem~\ref{th LOcally finite KPA of pullback}). For finite-vertex row-finite $k$-graphs without sources, we show that an isomorphism of talented monoids preserves the property of being locally finite. Equivalently, the corresponding graded Grothendieck groups, with their ordered $\mathbb Z[\mathbb Z^k]$-module structures, detect this class of algebras.

We obtain similar results for purely infinite simple Kumjian--Pask algebras. The relevant conditions can be expressed through the order structure of the talented monoid together with the $\mathbb Z^k$-action. In particular, the existence of loops with entrances plays a central role: under the appropriate aperiodicity and cofinality hypotheses, the geometry of such loops leads to infinite projections and hence to purely infinite behavior in the associated algebra (Theorem~\ref{th the TM criterion for purely infinite simplicity}).

Finally, we consider crossed-product Kumjian--Pask algebras. Here the $\mathbb Z^k$-action on the talented monoid becomes particularly significant. For a finite-vertex row-finite $k$-graph without sources, the crossed-product property can be expressed in terms of the action on the order-unit of $T_\Lambda$. We obtain a necessary condition in terms of the absence of entrances to loops and then identify an additional geometric condition which, together with this condition, is sufficient for the Kumjian--Pask algebra to be a crossed product (Theorem~\ref{th sufficient condition for crosspdt}). The higher-rank situation also reveals an important distinction from the theory of ordinary graphs. In the one-dimensional case, the corresponding geometric conditions have a much more rigid relationship with the algebraic structure. For $k$-graphs, however, the different grading directions can interact, and a condition which is sufficient in one setting need not be necessary in another. Our examples demonstrate this additional flexibility: a higher-rank graph may fail one of the natural geometric conditions while its talented monoid nevertheless exhibits the fixed-point behavior required for the associated Kumjian--Pask algebra to be a crossed product (Subsection~\ref{ssec crosspdt} together with the examples therein).

Our results suggest that understanding precisely how much geometric information is encoded in the $\mathbb Z^k$-ordered structure of $K_0^{\mathrm{gr}}$ may be a fruitful direction for the classification theory of higher-rank graph algebras as well as higher-dimensional symbolic dynamics. 

The paper is organized as follows. We begin in Section \ref{sec prelim} with a brief review of the concepts of higher-rank graphs and their talented monoids. In Section \ref{sec TM under product and pullbacks}, we investigate the behavior of talented monoids of higher-rank graphs under categorical products and pullbacks (Propositions \ref{pro TM of product hrg} and \ref{pro TM under pullbacks}). The main results of this paper are collected in Sections \ref{sec loop classification} and \ref{sec locally finite and crossed product}. In Subsection \ref{ssec loops}, we completely describe loops with or without entrances in a $k$-graph using the algebraic pre-order and the $\mathbb{Z}^k$-action of the associated talented monoid (Theorem \ref{th characterizing loops via TM}). Here we also define higher-dimensional analogues of Condition (L) and Condition (NE) of directed graphs and obtain talented monoid criteria for a $k$-graph to satisfy these conditions (Definition \ref{def condition L and NE in higher-dimension} and Corollary \ref{cor characterizing HL and HNE via TM}). Using the results of this subsection, we obtain, in Subsection \ref{ssec purely infinite simple}, a sufficient condition in terms of the talented monoid yielding purely infinite simple $C^*$-algebras and Kumjian--Pask algebras (Theorem \ref{th the TM criterion for purely infinite simplicity}). We prove that the existence of an entrance to a generalized cycle can be captured by the talented monoid (Lemma \ref{lem generalized cycle} and Theorem \ref{th existence of generalized cycle via TM}) and consequently, find conditions under which the Kumjian--Pask algebra cannot be ultramatricial (Proposition \ref{pro a necessary condition for KPA to be ultramatricial}). Section \ref{sec locally finite and crossed product} deals with locally finite and crossed product Kumjian--Pask algebras separately. In Subsection \ref{ssec locally finite}, we obtain necessary and sufficient conditions (in terms of the geometry of a $k$-graph as well as in the language of the talented monoid) for the Kumjian--Pask algebra to be locally finite as a $\mathbb{Z}^k$-graded algebra (Theorem \ref{th LOcally finite KPA of pullback}) which allows us to conclude that the talented monoid (equivalently, the graded $K$-theory) can detect this structural type of Kumjian--Pask algebras. Finally, in Subsection \ref{ssec crosspdt}, we define a higher-dimensional version of the \emph{equally distributed length \textup{(EDL)}} condition of a directed graph and establish that this, together with the higher-dimensional \emph{no exit} condition, is sufficient for the Kumjian--Pask algebra to be a crossed product (Theorem \ref{th sufficient condition for crosspdt}). Throughout this paper, we provide ample examples to illustrate our findings.

\section{Higher-rank graphs and their talented monoids}\label{sec prelim}
In this section, we collect the necessary background information on higher-rank graphs and their talented monoids that will be used in the article. The literature of higher-rank graphs and their analytic and discrete algebras is vast. For a comprehensive understanding of this subject, we refer the reader to \cite{KumjianPask,Pino,HMPS1,HMPS2}. 

Throughout this article, for a positive integer $k$, $\mathbb{Z}^k$ stands for the abelian group of all $k$-tuples of integers, partially with respect to the coordinate-wise ordering. The positive cone consisting of $k$-tuples of nonnegative integers is denoted by $\mathbb{N}^k$. For $i=1,2,\ldots,k$, $\mathbf{e}_i$ denotes the unit vector in $\mathbb{Z}^k$ with $1$ at the $i$-th position and $0$ elsewhere. The notation $\mathbb{Z}_{<0}^k$ denotes the collection of all $k$-tuples of nonpositive integers excluding the zero tuple. We use boldface characters to denote vectors ($k$-tuples) in $\mathbb{Z}^k$. The notation $\mathsf{F}$ stands for an arbitrary field.  

A \emph{higher-rank graph} or simply a $k$-\emph{graph} is defined as a countably small category $\Lambda$ together with a functor $d:\Lambda\longrightarrow \mathbb{N}^k$ satisfying the \emph{unique factorization criterion}: if $d(\lambda)=\N+\M$ for $\N,\M\in \mathbb{N}^k$, then there exist unique $\alpha,\beta\in \Lambda$ such that $d(\alpha)=\N$, $d(\beta)=\M$ and $\lambda=\alpha\beta$. 

The \emph{domain} and \emph{codomain} of $\lambda\in \Lambda$ are usually denoted by $s(\lambda)$ and $r(\lambda)$, respectively, and are often referred to as the source and range of $\lambda$. For any $\N\in \mathbb{N}^k$, $\Lambda^\N:=d^{-1}(\N)$. Using the unique factorization criterion, it is easy to check that $\lambda\in \Lambda^0$ if and only if $\lambda$ is an object of the category $\Lambda$. The elements of $\Lambda^0$ are often called \emph{vertices}, and morphisms of $\Lambda$ are called \emph{paths}. For any path $\lambda\in \Lambda$, $d(\lambda)$ is called the \emph{degree} of $\lambda$. If $d(\lambda)=(n_1,n_2,\ldots,n_k)$, then $|\lambda|:=\displaystyle{\sum_{i=1}^{k}}~ n_i$.

For $u,v\in \Lambda^0$ and $\N\in \mathbb{N}^k$, \[u\Lambda^\mathbf{n}:=r_\Lambda^{-1}(u)\cap \Lambda^\mathbf{n},~\Lambda^\mathbf{n} v:=\Lambda^\mathbf{n}\cap s_\Lambda^{-1}(v),~u\Lambda^\mathbf{n} v:=u\Lambda^\mathbf{n} \cap \Lambda^\mathbf{n} v.\] 

A $k$-graph is called \emph{row-finite} if $|u\Lambda^\N|<\infty$ for all $u\in \Lambda^0$ and $\N\in \mathbb{N}^k$. If in addition $u\Lambda^\N\neq \emptyset$ for all $u\in \Lambda^0$ and $\N\in \mathbb{N}^k$, then $\Lambda$ is called \emph{row-finite without sources}.

For $0\le \N<\M\le d(\lambda)$, $\lambda(\N,\M)$ denotes the unique path in $\Lambda$ such that $\lambda=\alpha \lambda(\N,\M)\beta$ for some $\alpha\in \Lambda^\N$ and $\beta\in \Lambda^{d(\lambda)-\M}$. When $\N=\M$, this path is obviously a vertex, which we denote by $\lambda(\N)$. In view of this notation, $\lambda(0)=r(\lambda)$ and $\lambda(d(\lambda))=s(\lambda)$. 

Let $\lambda,\mu$ be paths in $\Lambda$. A path $\tau\in \Lambda$ is called a \emph{minimal common extension} of $\lambda$ and $\mu$ if 

\begin{center}
    $d(\tau)=d(\lambda)\vee d(\mu)$, $\tau(0,d(\lambda))=\lambda$ and $\tau(0,d(\mu))=\mu$.
\end{center}
We write 
\begin{center}
$\MCE(\lambda,\mu):=\{\tau~|~\tau$ is a minimal common extension of $\lambda,\mu\}$   
\end{center}
and 
\begin{center}
$\Lambda^{\min}(\lambda,\mu):=\{(\alpha,\beta)~|~\lambda\alpha=\mu\beta\in~ \MCE(\lambda,\mu)\}$.
\end{center}
It is easy to see that there is bijection between the sets $\MCE(\lambda,\mu)$ and $\Lambda^{\min}(\lambda,\mu)$. 

A path $\mu$ for which $\mu(0)=\mu(d(\mu))$ (i.e., $s(\mu)=r(\mu)$) is called a \emph{loop}. A loop is called \emph{nontrivial} if it is not a vertex. 

We now recall the definition of entrance to a loop as it is going to be crucial in the context of this paper. In \cite{KPS}, it is defined for row-finite locally convex $k$-graphs. As a $k$-graph without sources is obviously locally convex, we can adopt this in the setting of row-finite $k$-graphs without sources.
\begin{dfn}\label{def loop with entrance}
Let $\Lambda$ be a row-finite $k$-graph without sources. Let $\mu$ be a loop in $\lambda$. A path $\lambda\in s(\mu)\Lambda$ is called an \emph{entrance} to $\mu$ if $d(\mu)\ge d(\lambda)$ and $\mu(0,d(\lambda))\neq \lambda$. 
\end{dfn}
Note that a loop can have an entrance only if it is nontrivial. Loops in a $k$-graph can be thought as the higher-rank analogue of closed paths in a directed graph. We now define the higher-dimensional analogue of a cycle in a directed graph. We believe that the following definition already exists in the literature of higher-rank graphs, possibly in some other name.  
\begin{dfn}\label{def cycle in a k-graph}
A nontrivial loop $\mu$ in a $k$-graph is called a \emph{cycle} if for $0\le \N<\M\le d(\mu)$, $\mu(\N)=\mu(\M)$ implies $\N=0$ and $\M=d(\mu)$.   
\end{dfn}
\begin{dfn}\label{def length and unicolor path}
Let $i\in \{1,2,\ldots,k\}$. A path $\mu\in \Lambda$ is called \emph{unicolor in degree $\mathbf{e}_i$} if $d(\mu)=|\mu|\mathbf{e}_i$. unicolor cycles are defined analogously.
\end{dfn}
The importance of loops with entrances is transpired from the following result due to Kumjian, Pask and Sims \cite{KPS}, which gives a sufficient condition for the $C^*$-algebra of a higher-rank graph to be purely infinite simple. 
\begin{thm}\label{th pis of C*-algebra}
\textup{(\cite[Theorem 2]{KPS})} Let $\Lambda$ be an aperiodic, row-finite $k$-graph without sources. Suppose that every vertex of $\Lambda$ can be reached from a loop with an entrance. Then every hereditary subalgebra has an infinite projection. Hence, if $\Lambda$ is cofinal, then $C^*(\Lambda)$ is purely infinite simple.
\end{thm} 

The notion of pure infiniteness goes beyond the analytic set-up of $C^*$-algebras. In \cite{AGP}, Ara, Goodearl and Pardo introduced purely infinite simplicity for discrete rings and then in \cite{AGPM}, the notion is extended for rings which are not necessarily simple. We recall some preliminaries from \cite{AGPM} before stating the definition of purely infinite simple rings. 

Let $R$ be a ring. Two idempotents $e,f\in R$ are called \emph{Murray-von Neumann equivalent} if there exist $x,y\in R$ such that $e=xy$ and $f=yx$. This is expressed by the notation $e\sim f$. The idempotents $e,f$ are called \emph{orthogonal} if $ef=fe=0$. In this case, their sum $e+f$ is also an idempotent. There is a partial order $\leqslant$ on the set of idempotents of $R$, namely $e\leqslant f$ if and only if $ef=fe=e$. If $e\leqslant f$ and $e\neq f$, then we write $e< f$. An idempotent $e\in R$ is called \emph{infinite} if there is an idempotent $f$ such that $e\sim f< e$. This is equivalent to saying that $eR\cong eR\oplus gR$ for some nonzero idempotent $g$ orthogonal to $e$. A simple ring $R$ is called \emph{purely infinite simple} if every nonzero one sided ideal (left or right) contains an infinite nonzero idempotent. 

A $C^*$-algebra $\mathcal{A}$ is called $AF$ if there exist a countably infinite family of finite-dimensional $C^*$-subalgebras $\{\mathcal{A}_n~|~n\in \mathbb{N}\}$ such that $\mathcal{A}_n\subseteq \mathcal{A}_{n+1}$ for all $n\in \mathbb{N}$ and $\mathcal{A}=\overline{\displaystyle{\bigcup_{n\in \mathbb{N}}}~\mathcal{A}_n}$. Evans and Sims in \cite{Evans} considered the question that when the $C^*$-algebra of a higher-rank graph is $AF$ and established some interesting geometric criteria, which prohibit the $C^*$-algebra to be $AF$. For this, they defined and used the concept of generalized cycles, which we now state in the setting of row-finite $k$-graphs.
\begin{dfn}
(\cite[Definition 3.1]{Evans}) Let $\Lambda$ be a row-finite $k$-graph. Suppose $\mu,\nu$ are distinct paths in $\Lambda$ with $s(\mu)=s(\nu)$ and $r(\mu)=r(\nu)$. Then the pair $(\mu,\nu)$ is called a \emph{generalized cycle} if for any $\tau\in s(\mu)\Lambda$, $\MCE(\mu\tau,\nu)\neq \emptyset$.     
\end{dfn}

It is clear that for any loop or cycle $\mu$ in $\Lambda$, the pair $(\mu,s(\mu))$ is a generalized cycle; however, $(s(\mu),\mu)$ is not a generalized cycle if there exists $\nu\in s(\mu)\Lambda\setminus \{\mu\}$ such that $d(\nu)=d(\mu)$. This implies that the definition of a generalized cycle is not symmetric. If $(\mu,\nu)$ is a generalized cycle, then a path $\tau\in s(\nu)\Lambda$ is called an \emph{entrance} to $(\mu,\nu)$ if $\MCE(\mu,\nu\tau)=\emptyset$.    

As any finite dimensional $C^*$-algebra is essentially a finite direct sum of matrix algebras over $\mathbb{C}$, there is an obvious discretization of the concept of $AF$ $C^*$-algebras, called ultramatricial algebras. More precisely, an $\mathsf{F}$-algebra $A$ is called \emph{matricial} if it is isomorphic to a finite direct sum of finite dimensional matrix algebras over $\mathsf{F}$. $A$ is called \emph{ultramatricial} if it isomorphic to a direct limit of a countable sequence of matricial $\mathsf{F}$-algebras. It can be shown that any ultramatricial algebra is directly finite and consequently, contains no infinite idempotents.

In \cite{Pino}, Aranda Pino et al. defined the discrete version of higher-rank graph $C^*$-algebras \cite{KumjianPask}, commonly known as \emph{Kumjian--Pask algebras}. 

\begin{dfn}\label{def KP-algebra}
({\bf Kumjian--Pask algebra}) Let $\Lambda$ be a row-finite $k$-graph without sources. With each path $\lambda\in \Lambda\setminus \Lambda^0$, associate a formal symbol $\lambda^*$. Then the \emph{Kumjian--Pask algebra} of $\Lambda$ over $\mathsf{F}$ is denoted by $\KP_\mathsf{F}(\Lambda)$ and is defined to be the universal associative $\mathsf{F}$-algebra generated by formal symbols $p_v,s_\lambda,s_{\lambda^*}$; $v\in \Lambda^0,\lambda\in \Lambda\setminus \Lambda^0$ subject to the following relations:

$(KP1)$ $p_up_v=\delta_{u,v}p_u$ for all $u,v\in \Lambda^0$;

$(KP2)$ $s_\lambda s_\mu=s_{\lambda\mu},$ $s_{\mu^*}s_{\lambda^*}=s_{(\lambda\mu)^*}$ for all $\lambda,\mu\in \Lambda^{\neq 0}$ with $s(\lambda)=r(\mu)$, and

$p_{r(\lambda)}s_\lambda=s_\lambda=s_\lambda p_{s(\lambda)},$ $p_{s(\lambda)}s_{\lambda^*}=s_{\lambda^*}=s_{\lambda^*}p_{r(\lambda)}$ for all $\lambda\in \Lambda^{\neq 0}$;

$(KP3)$ $s_{\lambda^*} s_\mu=\delta_{\lambda,\mu}p_{s(\lambda)}$ for all $\lambda,\mu\in \Lambda^{\neq 0}$ with $d(\lambda)=d(\mu)$;

$(KP4)$ $p_v=\displaystyle{\sum_{\lambda\in v\Lambda^\N}} s_\lambda s_{\lambda^*}$ for all $v\in \Lambda^0$ and for all $\N\in \mathbb{N}^k\setminus \{0\}$.
\end{dfn}

In view of relation $(KP1)$, $p_v$ are orthogonal idempotents. Also, $s_\lambda s_{\lambda^*}$ are idempotents by $(KP2)$ and $(KP3)$. 

Despite the close similarity in the definition of Leavitt path algebras and Kumjian--Pask algebras, the algebraic structure of Kumjian--Pask algebras is much more complex than that of Leavitt path algebras. This is mainly due to the delicate geometry of the underlying higher-rank graph arising from the factorization of paths having multidimensional degrees. The class of Kumjian--Pask algebras not only covers Leavitt path algebras but also covers some algebras which cannot be realized as Leavitt path algebra of a directed graph. For example, if we consider the canonical $2$-graph $\Lambda:=\mathbb{N}^2$ with the identity morphism playing the role of the functor $d$, then its Kumjian--Pask algebra is isomorphic to the algebra of Laurent polynomials in two commuting variables, which is not Leavitt path algebra of any directed graph (see \cite[Example 7.1]{Pino}).  

The algebra $\KP_\mathsf{F}(\Lambda)$ is a $\mathbb{Z}^k$-graded algebra, where for any $\N\in \mathbb{Z}^k$, \[\KP_\mathsf{F}(\Lambda)_\N:=\lspan_\mathsf{F}\left(\{s_\lambda s_{\mu^*}~|~\lambda,\mu\in \Lambda~\text{and}~d(\lambda)-d(\mu)=\N\}\right).\] Following \cite{AAS}, we say that a $\Gamma$-graded $\mathsf{F}$-algebra $A=\displaystyle{\bigoplus_{\gamma\in \Gamma}}~A_\gamma$ is \emph{locally finite} if each homogeneous component $A_\gamma$ is a finite-dimensional $F$-subspace of $A$. In \cite{AAS}, the authors gave a complete description of locally finite Leavitt path algebras of directed graphs. One of the goals of this article is to extend their result to the higher-dimensional setting of Kumjian--Pask algebras. 

Another important class of graded algebras consists of crossed product algebras. Let $\Gamma$ be a group and $A=\displaystyle{\bigoplus_{\gamma\in \Gamma}}$ a $\Gamma$-graded unital algebra (or ring). Then $A$ is called a \emph{crossed product} if each homogeneous component $A_\gamma$ contains a unit (invertible element).

We finish this preliminary section by reviewing talented monoids of higher-rank graphs \cite{HMPS1}. For this we first recall the notion of $\Gamma$-monoids. Let $\Gamma$ be a group and $M$ a commutative monoid. Then $M$ is called a \emph{$\Gamma$-monoid} if there is a group homomorphism $\eta:\Gamma\longrightarrow \Aut(M)$. For $x\in M$ and $\gamma\in \Gamma$, we usually write $^\gamma x$ to denote the element $\phi(\gamma)(x)$. Every commutative monoid $M$ is pre-ordered with respect to the \emph{algebraic pre-order} $\le$ defined by $x\le y$ if and only if $y=x+z$ for some $z\in M$. An element $\epsilon$ of a $\Gamma$-monoid $M$ is called an \emph{order-unit} if for any $x\in M$, there exist positive integers $n_1,n_2,\ldots,n_t$ and $\gamma_1,\gamma_2,\ldots,\gamma_t\in \Gamma$ such that $x\le \displaystyle{\sum_{i=1}^{t}} n_i(^{\gamma_i}\epsilon)$.

For the rest of this section, we fix a row-finite $k$-graph $\Lambda$. By $\mathbb{F}_\Lambda$, we denote the free commutative monoid generated by $\Lambda^0$. For each $i=1,2,\ldots,k$, define a binary relation $\longrightarrow_{i}$ on $\mathbb{F}_\Lambda\setminus \{0\}$ as follows: if $x,y\in \mathbb{F}_\Lambda$ and $x=\displaystyle{\sum_{t=1}^{n}} v_t$ with $v_t\in \Lambda^0$, then $x\longrightarrow_i y$ if and only if \[y=\displaystyle{\sum_{\substack{1\le t\le n\\t\neq j}}} v_t+\displaystyle{\sum_{\alpha\in v_j\Lambda^{\mathbf{e}_i}}}s(\alpha)\] for some $j\in \{1,2,\ldots,n\}$. Let $\longrightarrow$ denote the reflexive and transitive closure of the $k$ relations $\longrightarrow_i$ and $\sim$ denote the congruence on $\mathbb{F}_\Lambda$ generated by $\longrightarrow$. Then the \emph{$k$-graph monoid} of $\Lambda$ is defined as $M_\Lambda:=\mathbb{F}_\Lambda/\sim$ (see \cite{HMPS1}). This is a conical monoid and when $\Lambda$ is without sources, it has a confluence property as described in \cite[Lemma 3.5]{HMPS1}. 

\begin{dfn}\label{def talented monoid}
({\bf Talented monoid}) The \emph{talented monoid} of $\Lambda$ is denoted by $T_\Lambda$ and is defined as the quotient of the free commutative monoid generated by $\{v(n)~|~v\in \Lambda^0, n\in \mathbb{Z}^k\}$ modulo the congruence defined by the following relations:
\begin{equation}\label{TM equation}
    v(n)=\displaystyle{\sum_{\alpha\in v\Lambda^{e_i}}} s(\alpha)(n+e_i)
\end{equation}
for each $v\in \Lambda^0$, $n\in \mathbb{Z}^k$ and $i\in \{1,2,...,k\}$ with $v\Lambda^{e_i}\neq \emptyset$.   
\end{dfn}
When $\Lambda$ has no sources, then the set of relations (\ref{TM equation}) is equivalent to the relations 
\begin{equation*}
    v(n)=\displaystyle{\sum_{\alpha\in v\Lambda^m}} s(\alpha)(n+m).
\end{equation*}

The monoid $T_\Lambda$ comes with a natural action of the group $\mathbb{Z}^k$ which is defined on the generators as follows: $^\Q v(\N):=v(\N+\Q)$. This action makes $T_\Lambda$ a $\mathbb{Z}^k$-monoid, and when $|\Lambda^0|<\infty$, $\epsilon_\Lambda:=\displaystyle{\sum_{v\in \Lambda^0}} v(0)$ is an order-unit of this $\mathbb{Z}^k$-monoid. 

The $\mathbb{Z}^k$-action distinguishes the talented monoid from the ordinary $k$-graph monoid and is very much useful in understanding the geometry of the underlying $k$-graph as evident from the results in \cite{HMPS1} and will also be in the development of this paper. The most important fact about the talented monoid is that it recovers the graded Grothendieck group of the Kumjian--Pask algebra of $\Lambda$ directly from the geometric data. To be precise, $T_\Lambda$ is $\mathbb{Z}^k$-isomorphic to the monoid $\mathcal{V}^{\gr}(\KP_\mathsf{F}(\Lambda))$ where the isomorphism sends $v(\N)$ to $[\KP_\mathsf{F}(\Lambda)p_v (\N)]$ (see \cite[Lemma 5.6 \& Corollary 6.6]{AHLS} and \cite[Proposition 3.14]{HMPS1}). From this, it follows that $T_\Lambda$ is cancellative and can be realized as the positive cone of the $\mathbb{Z}[\mathbb{Z}^k]$-module $K_0^{\gr}(\KP_\mathsf{F}(\Lambda))$. 

We recall the definition of the skew-product of a $k$-graph $\Lambda$ with the group $\mathbb{Z}^k$. It is denoted by $\overline{\Lambda}$ and is a $k$-graph with underlying set $\Lambda\times \mathbb{Z}^k$, structure maps: $\overline{r}((\lambda,\N)):=(r(\lambda),\N)$, $\overline{s}((\lambda,\N)):=(s(\lambda),d(\lambda)+\N)$, composition: $(\lambda,\N)(\mu,d(\lambda)+\N):=(\lambda\mu,\N)$ (provided $r(\mu)=s(\lambda)$) and degree map $\overline{d}((\lambda,\N)):=d(\lambda)$. If we consider the elements $(\lambda,\N)$; $\lambda\in \Lambda$ as of state $\N$, then it is clear that for any path $(\lambda,\N)$, the state of its source is ahead of the state of its range. The $k$-graph $\overline{\Lambda}$ has the following properties, which we will exploit in some of our proofs of this paper.

$(i)$ $\overline{\Lambda}$ is always countably infinite.

$(ii)$ $\overline{\Lambda}$ has no loop.

$(iii)$ $\overline{\Lambda}$ repeats the same pattern in each state: if $(\lambda,\N)$ is path from $(u,d(\Lambda)+\N)$ to $(v,\N)$ then for any $\M\in \mathbb{Z}^k$, $(\lambda,\N+\M)$ is a path from $(u,d(\lambda)+\N+\M)$ to $(v,\N+\M)$.

In \cite[Theorem 3.16$(ii)$]{HMPS1}, it is shown that there is a $\mathbb{Z}^k$-monoid isomorphism $\phi:T_\Lambda\longrightarrow M_{\overline{\Lambda}}$, which sends $v(\N)$ to $(v,\N)$.

\section{Talented monoid under some basic operations on higher-rank graphs}\label{sec TM under product and pullbacks}
In this section, we discuss two basic operations on higher-rank graphs, namely categorical products and pullbacks. In Propositions \ref{pro TM of product hrg} and \ref{pro TM under pullbacks}, we investigate how the talented monoids of $k$-graphs react under these operations.

\subsection{Under categorical products}\label{ssec categorical product} Let $(\Lambda_1,d_1)$ and $(\Lambda_2,d_2)$ be higher-rank graphs of ranks $k$ and $\ell$ respectively. Consider the product category $\Lambda_1\times \Lambda_2$. Then it can be given the structure of a $(k+\ell)$-graph by defining the degree functor as 
\begin{align*}
d_1\times d_2:\Lambda_1\times \Lambda_2&\longrightarrow \mathbb{N}^{k+\ell}\\
(\lambda_1,\lambda_2)& \longmapsto \iota_1(d_1(\lambda_1))+\iota_2(d_2(\lambda_2)),
\end{align*}
where $\iota_1,\iota_2$ are the respective canonical inclusions of $\mathbb{Z}^k$ and $\mathbb{Z}^\ell$ into $\mathbb{Z}^k\times \mathbb{Z}^\ell=\mathbb{Z}^{k+\ell}$. Let $\pi_1,\pi_2$ be the canonical projections of $\mathbb{Z}^{k+\ell}=\mathbb{Z}^k\times \mathbb{Z}^\ell$ onto $\mathbb{Z}^k\times \{0\}$ and $\{0\}\times \mathbb{Z}^\ell$ respectively. We are interested to know the relationship between the talented monoid of $\Lambda_1\times \Lambda_2$ and the talented monoids of $\Lambda_1$, $\Lambda_2$. 

\begin{prop}\label{pro TM of product hrg}
Let $(\Lambda_1,d_1)$ and $(\Lambda_2,d_2)$ be $k$- and $\ell$-graph respectively. Suppose both are row-finite and without sources. Then 

$(i)$ the tensor product $T_{\Lambda_1}\otimes T_{\Lambda_2}$ is a $\mathbb{Z}^{k+\ell}$-monoid, where \[^\mathbf{q}(v(\M)\otimes w(\N))=v(\M+\iota_1^{-1}(\pi_1(\mathbf{q})))\otimes w(\N+\iota_2^{-1}(\pi_2(\mathbf{q}))).\]

$(ii)$ $T_{\Lambda_1\times \Lambda_2}$ and $T_{\Lambda_1}\otimes T_{\Lambda_2}$ are isomorphic as $\mathbb{Z}^{k+\ell}$-monoids. 
\end{prop}
\begin{proof}
$(i)$ For each $\mathbf{q}\in \mathbb{Z}^{k+\ell}$, define a map 
\begin{align*}
\eta_\mathbf{q}:T_{\Lambda_1}\times T_{\Lambda_2}&\longrightarrow T_{\Lambda_1}\otimes T_{\Lambda_2}\\
\left(\displaystyle{\sum_{i=1}^{p}} v_i(\M_i),\displaystyle{\sum_{j=1}^{q}} w_j(\N_j)\right)&\longmapsto \displaystyle{\sum_{i=1}^{p}}\left(\displaystyle{\sum_{j=1}^{q}} v_i(\M_i+\iota_1^{-1}(\pi_1(\mathbf{q})))\otimes w_j(\N_j+\iota_2^{-1}(\pi_2(\mathbf{q})))\right). 
\end{align*}
One can then check that the map $\eta_\mathbf{q}$ is a well-defined bilinear map. As a consequence, it extends to a unique monoid homomorphism $\overline{\eta_\mathbf{q}}:T_{\Lambda_1}\otimes T_{\Lambda_2}\longrightarrow T_{\Lambda_1}\otimes T_{\Lambda_2}$. Obviously, $\overline{\eta_\mathbf{q}}$ is invertible with inverse homomorphism $\overline{\eta_{-\mathbf{q}}}$. Hence, $\mathbb{Z}^{k+\ell}$ acts on $T_{\Lambda_1}\otimes T_{\Lambda_2}$ by automorphisms, making it a $\mathbb{Z}^{k+\ell}$-monoid. It is evident that \[^\mathbf{q}(v(\M)\otimes w(\N))=\overline{\eta_\mathbf{q}}(v(\M)\otimes w(\N))=\eta_\mathbf{q}(v(\M),w(\N))=v(\M+\iota_1^{-1}(\pi_1(\mathbf{q})))\otimes w(\N+\iota_2^{-1}(\pi_2(\mathbf{q}))).\]

$(ii)$ Consider the following map: 
\begin{align*}
\phi:T_{\Lambda_1}\times T_{\Lambda_2}&\longrightarrow T_{\Lambda_1\times \Lambda_2}\\
\left(\displaystyle{\sum_{i=1}^{p}} v_i(\M_i),\displaystyle{\sum_{j=1}^{q}} w_j(\N_j)\right)&\longmapsto \displaystyle{\sum_{i=1}^{p}}\left(\displaystyle{\sum_{j=1}^{q}} (v_i,w_j)(\iota_1(\M_i)+\iota_2(\N_j))\right).
\end{align*}
Let $(v(\M),w(\N))\in T_{\Lambda_1}\times T_{\Lambda_2}$. In $T_{\Lambda_1}$, we have $v(\M)=\displaystyle{\sum_{\alpha\in v\Lambda_1^{\mathbf{q}}}} s(\alpha)(\M+\mathbf{q})$ for all $\mathbf{q}\in \mathbb{N}^k$; whereas in $T_{\Lambda_2}$, $w(\N)=\displaystyle{\sum_{\beta\in w\Lambda_2^{\mathbf{t}}}} s(\beta)(\N+\mathbf{t})$ for all $\mathbf{t}\in \mathbb{N}^\ell$. Note that 
\begin{align*}
&~(v,w)(\Lambda_1\times \Lambda_2)^{\iota_1(\mathbf{q})+\iota_2(\mathbf{t})}\\
=&~\{(\alpha,\beta)\in \Lambda_1\times \Lambda_2~|~r(\alpha)=v,r(\beta)=w,~\text{and}~\iota_1(d_1(\alpha))+\iota_2(d_2(\beta))=\iota_1(\mathbf{q})+\iota_2(\mathbf{t})\}\\
=&~\{\alpha\in \Lambda_1~|~r(\alpha)=v,\iota_1(d_1(\alpha))=\iota_1(\mathbf{q})\}\times \{\beta\in \Lambda_2~|~r(\beta)=w,\iota_2(d_2(\beta))=\iota_2(\mathbf{t})\}\\
=&~v\Lambda_1^{\mathbf{q}}\times w\Lambda_2^{\mathbf{t}}
\end{align*} since $\iota_1,\iota_2$ are injective. In view of this fact, we have 
\begin{align*}
&\phi\left(\displaystyle{\sum_{\alpha\in v\Lambda_1^{\mathbf{q}}}} s(\alpha)(\M+\mathbf{q}),\displaystyle{\sum_{\beta\in w\Lambda_2^{\mathbf{t}}}} s(\beta)(\N+\mathbf{t})\right)\\
=&\displaystyle{\sum_{\alpha\in v\Lambda_1^{\mathbf{q}},\beta\in w\Lambda_2^{\mathbf{t}}}} (s(\alpha),s(\beta))(\iota_1(\M+\mathbf{q})+\iota_2(\N+\mathbf{t}))\\
=&\displaystyle{\sum_{(\alpha,\beta)\in (v,w)(\Lambda_1\times \Lambda_2)^{\iota_1(\mathbf{q})+\iota_2(\mathbf{t})}}} s((\alpha,\beta))(\iota_1(\M)+\iota_2(\N)+\iota_1(\mathbf{q})+\iota_2(\mathbf{t}))\\
=&(v,w)(\iota_1(\M)+\iota_2(\N))=\phi(v(\M),w(\N)), 
\end{align*}
which shows that $\phi$ is well-defined. By definition, $\phi$ is bilinear. Hence, it extends to a unique monoid homomorphism $\Phi:T_{\Lambda_1}\otimes T_{\Lambda_2}\longrightarrow T_{\Lambda_1\times \Lambda_2}$. Now, we define a map in the other direction. Define $\Psi:T_{\Lambda_1\times \Lambda_2}\longrightarrow T_{\Lambda_1}\otimes T_{\Lambda_2}$ on generators as \[\Psi((v,w)(\mathbf{p})):=v(\iota_1^{-1}(\pi_1(\mathbf{p})))\otimes w(\iota_2^{-1}(\pi_2(\mathbf{p})))\] and then extend linearly to all of $T_{\Lambda_1\times \Lambda_2}$. Again, we check that $\Psi$ is well-defined. For any $\mathbf{q}\in \mathbb{N}^{k+\ell}$, using the same fact as used earlier, we have
\allowdisplaybreaks
{
\begin{align*}
&\Psi\left(\displaystyle{\sum_{(\lambda,\mu)\in (v,w)(\Lambda_1\times \Lambda_2)^{\mathbf{q}}}} s((\lambda,\mu))(\mathbf{p}+\mathbf{q})\right)\\
=&\displaystyle{\sum_{(\lambda,\mu)\in (v,w)(\Lambda_1\times \Lambda_2)^{\mathbf{q}}}} s(\lambda)(\iota_1^{-1}(\pi_1(\mathbf{p}+\mathbf{q})))\otimes s(\mu)(\iota_2^{-1}(\pi_2(\mathbf{p}+\mathbf{q})))\\
=&\displaystyle{\sum_{(\lambda,\mu)\in v\Lambda_1^{\iota_1^{-1}(\pi_1(\mathbf{q}))} \times w\Lambda_2^{\iota_2^{-1}(\pi_2(\mathbf{q}))}}} s(\lambda)(\iota_1^{-1}(\pi_1(\mathbf{p}))+\iota_1^{-1}(\pi_1(\mathbf{q})))\otimes s(\mu)(\iota_2^{-1}(\pi_2(\mathbf{p}))+\iota_2^{-1}(\pi_2(\mathbf{q})))\\
=&\left(\displaystyle{\sum_{\lambda\in v\Lambda_1^{\iota_1^{-1}(\pi_1(\mathbf{q}))}}} s(\lambda)(\iota_1^{-1}(\pi_1(\mathbf{p}))+\iota_1^{-1}(\pi_1(\mathbf{q})))\right)\otimes \left(\displaystyle{\sum_{\mu\in w\Lambda_2^{\iota_2^{-1}(\pi_2(\mathbf{q}))}}} s(\mu)(\iota_2^{-1}(\pi_2(\mathbf{p}))+\iota_2^{-1}(\pi_2(\mathbf{q})))\right)\\
=&~v(\iota_1^{-1}(\pi_1(\mathbf{p})))+w(\iota_2^{-1}(\pi_2(\mathbf{p})))=\Psi((v,w)(\mathbf{p})).
\end{align*}
}
Therefore, $\Psi$ is a well-defined monoid homomorphism. It is now easy to observe that $\Phi$ and $\Psi$ are mutually inverse homomorphisms. Thus, we have the desired isomorphism. Finally, for any generator $(v,w)(\mathbf{p})$ of $T_{\Lambda_1\times \Lambda_2}$, 
\begin{align*}
\Psi(^\mathbf{q}(v,w)(\mathbf{p}))&=\Psi((v,w)(\mathbf{p}+\mathbf{q}))\\
&=v(\iota_1^{-1}(\pi_1(\mathbf{p}+\mathbf{q})))\otimes w(\iota_2^{-1}(\pi_2(\mathbf{p}+\mathbf{q})))\\
&=^\mathbf{q}(v(\iota_1^{-1}(\pi_1(\mathbf{p})))\otimes w(\iota_2^{-1}(\pi_2(\mathbf{p})))\\
&=^\mathbf{q}\Psi((v,w)(\mathbf{p})).    
\end{align*}
Thus, the isomorphism $\Psi$ preserves the $\mathbb{Z}^{k+\ell}$-action and the result follows. 
\end{proof}
The following example illustrates the usage of the above proposition.
\begin{example}
Consider a $2$-graph $\Lambda$ with the following $1$-skeleton: 
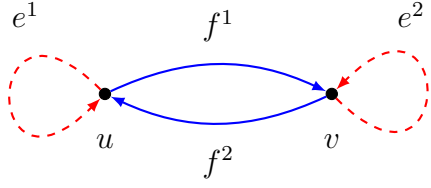
\[
\begin{tikzpicture}[scale=1.5]
\node[inner sep=1.5pt, circle,draw,fill=black] (A) at (0,0) {}
edge[-latex, red,thick, loop, dashed, out=135, in=225, looseness=8, distance=1.5cm] (A);

\node[inner sep=1.5pt, circle,draw,fill=black] (B) at (2,0) {}
edge[-latex, red,thick, loop, dashed, out=-45, in=45, looseness=8, distance=1.5cm] (B);

\path[->, blue, >=latex,thick] (A) edge [bend left=25] node[above=0.05cm]{} (B);
\path[->, blue, >=latex,thick] (B) edge [bend left=25] node[above=0.05cm]{} (A);

\node at (0,-0.42) {$u$};
\node at (2,-0.42) {$v$};

\node at (1,0.6) {$f^1$};
\node at (1,-0.6) {$f^2$};

\node at (-0.7,0.7) {$e^1$};
\node at (2.7,0.7) {$e^2$};
\end{tikzpicture}
\]

The factorizations are uniquely determined by the skeleton. Now consider the following directed graphs:
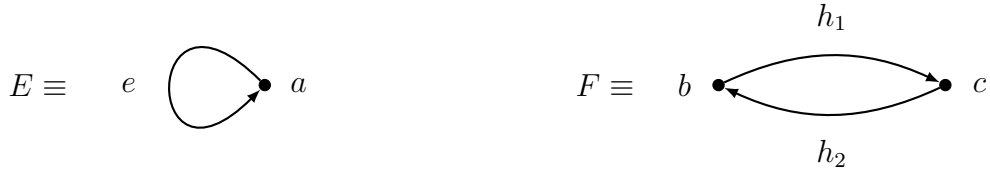
\[
\begin{tikzpicture}[scale=1.5]
\node[inner sep=1.5pt, circle,draw,fill=black] (A) at (0,0) {}
edge[-latex, black,thick, loop, out=135, in=225, looseness=8, distance=1.5cm] (A);  

\node[inner sep=1.5pt, circle,draw,fill=black] (B) at (4,0) {};	
\node[inner sep=1.5pt, circle,draw,fill=black] (C) at (6,0) {};

\path[->, >=latex,thick] (B) edge [bend left=25] node[above=0.05cm]{} (C);
\path[->, >=latex,thick] (C) edge [bend left=25] node[above=0.05cm]{} (B);

\node at (0.3,0) {$a$};

\node at (3.7,0) {$b$};
\node at (6.3,0) {$c$};

\node at (-1.2,0) {$e$};
\node at (5,0.6) {$h_1$};
\node at (5,-0.6) {$h_2$};

\node at (-2,0) {$E\equiv$};
\node at (3,0) {$F\equiv$};
\end{tikzpicture}
\]
It is not hard to see that the product of the $1$-graphs $E^*$ and $F^*$ is isomorphic to the $2$-graph $\Lambda$. Indeed the correspondence $(a,b)\longmapsto u$, $(a,c)\longmapsto v$, $(e,b)\longmapsto e^1$, $(e,c)\longmapsto e^2$, $(a,h_1)\longmapsto f^2$ and $(a,h_2)\longmapsto f^1$ extends to a $2$-graph isomorphism between $E^*\times F^*$ and $\Lambda$. Therefore, in view of Proposition \ref{pro TM of product hrg}, we have \[T_\Lambda\cong T_{E^*\times F^*}\cong T_{E^*}\otimes T_{F^*}\cong T_E\otimes T_F\cong \mathbb{N}\otimes (\mathbb{N}\oplus \mathbb{N})\cong \mathbb{N}\oplus \mathbb{N}.\] 
\end{example}
\subsection{Under pullbacks}\label{ssec pullback} Given an $\ell$-graph $\Lambda$ and monoid homomorphism $f:\mathbb{N}^k\longrightarrow \mathbb{N}^\ell$, recall that one may form the \emph{pullback $k$-graph} $f^*(\Lambda)$ (see \cite[Definition 1.9]{KumjianPask}) where the underlying set is the fibered product $\Lambda {{}_{d}\times_{f}} \mathbb{N}^k$. The range, source and degree maps are defined as $r((\lambda,\N)):=(r(\lambda),\N)$, $s((\lambda,\N)):=(s(\lambda),\N)$ and $d((\lambda,\N)):=\N$.

If $f:\mathbb{N}^k\longrightarrow \mathbb{N}^\ell$ is a surjective homomorphism and $\Lambda$ is an $\ell$-graph, then in the proof of \cite[Proposition 4.8(ii)]{HMPS1}, it was established that there is a surjective monoid homomorphism $\rho:T_{f^*(\Lambda)}\longrightarrow T_\Lambda$. In the proposition below, we strengthen this by showing that $\rho$ is indeed an isomorphism.   

To state the following proposition, we define \emph{equivariant homomorphisms} between monoids with group actions. In what follows, a pair $(\Gamma,M)$ indicates a commutative monoid $M$ equipped with a $\Gamma$-action (which we often called a $\Gamma$-monoid). An \emph{equivariant homomorphism} between two such pairs $(\Gamma,M)$ and $(\Delta,N)$ is defined as a pair $(f,\rho)$, where $f:\Gamma\longrightarrow \Delta$ is a group homomorphism and $\rho:M\longrightarrow N$ is a monoid homomorphism such that $\rho(^\gamma x)=~^{f(\gamma)}\rho(x)$ for all $\gamma\in \Gamma$ and $x\in M$. By \textbf{Grp-Mon}, we denote the category where objects are commutative monoids equipped with group actions and morphisms are equivariant homomorphisms.     

\begin{prop}\label{pro TM under pullbacks}
Let $\Lambda$ be a row-finite $\ell$-graph without sources and $f:\mathbb{N}^k\longrightarrow \mathbb{N}^\ell$ a monoid homomorphism. Then there exists a monoid homomorphism $\rho:T_{f^*(\Lambda)}\longrightarrow T_{\Lambda}$ such that the diagram 

\begin{center}
\begin{tikzcd}
(\mathbb{Z}^k,T_{f^*(\Lambda)})\arrow[rr, "{(\text{id},\rho)}"] \arrow[dr, "{(f,\rho)}"'] & & (\mathbb{Z}^k,T_\Lambda) \arrow[dl, "{(f,\text{id})}"] \\
& (\mathbb{Z}^\ell, T_\Lambda) &
\end{tikzcd}
\end{center}
commutes in the category \textup{\textbf{Grp-Mon}}, where in $(\mathbb{Z}^k,T_\Lambda)$ we view $T_\Lambda$ as a $\mathbb{Z}^k$-monoid with respect to the action $^\Q a:=~^{f(\Q)}a$ for all $\Q\in \mathbb{Z}^k$ and $a\in T_\Lambda$. If $f$ is surjective, then $\rho$ is an isomorphism. 
\end{prop}
\begin{proof}
Define $\rho:T_{f^*(\Lambda)}\longrightarrow T_\Lambda$ by $\rho(v(\N)):=v(f(\N))$ and then extend linearly to all of $T_{f^*(\Lambda)}$. It follows from the proof of \cite[Proposition 4.8(ii)]{HMPS2} that $\rho$ is a well-defined monoid homomorphism. For any $\N\in \mathbb{Z}^k$ and $a=\displaystyle{\sum_{i=1}^{t}}v_i(\M_i)\in T_{f^*(\Lambda)}$, \[\rho(^\N a)=\rho\left(\displaystyle{\sum_{i=1}^{t}} v_i(\M_i+\N)\right)=\displaystyle{\sum_{i=1}^{t}}v_i(f(\M_i)+f(\N))=~^{f(\N)}\rho(a),\] which shows that $(f,\rho)$ is equivariant. The commutativity of the diagram is now obvious. Suppose $f$ is surjective and choose any $b=\displaystyle{\sum_{i=1}^{t}}w_i(\M_i)\in T_\Lambda)$. Then there exist $\N_i\in \mathbb{Z}^k$, $i=1,2,\ldots,t$ such that $f(\N_i)=\M_i$. Clearly, $\rho\left(\displaystyle{\sum_{i=1}^{t}} v_i(\N_i)\right)=b$, showing that $\rho$ is surjective. Thus, it only remains to show that $\rho$ is injective. So assume that $x,y\in T_{f^*(\Lambda)}$ such that $\rho(x)=\rho(y)$. Using the fact that $f^*(\Lambda)$ has no sources (which is evident since $\Lambda$ has no sources), we can write $x=\displaystyle{\sum_{i=1}^{p}} v_i(\N)$, $y=\displaystyle{\sum_{j=1}^{q}}w_j(\N)$ for some vertices $v_i,w_j\in f^*(\Lambda)^0$ and some $\N\in \mathbb{N}^k$. Now, we have 
\begin{equation*}
\displaystyle{\sum_{i=1}^{p}}v_i(f(\N))=\displaystyle{\sum_{j=1}^{q}}w_j(f(\N))
\end{equation*}
in $T_\Lambda$. Transferring this equality to the $\ell$-graph monoid $M_{\overline{\Lambda}}$ and using the confluence lemma (\cite[Lemma 3.5]{HMPS1}), we have a $\gamma\in \mathbb{F}_{\overline{\Lambda}}\setminus \{0\}$ such that \[\displaystyle{\sum_{i=1}^{p}}(v_i,f(\N))\longrightarrow \gamma,~\displaystyle{\sum_{j=1}^{q}}(w_j,f(\N))\longrightarrow \gamma.\] By flowing using the relations $\longrightarrow_{\mathbf{e}_i}$, if required, we can always write $\gamma=\displaystyle{\sum_{l=1}^{r}} (u_l,\Q)$ for some $\Q\in \mathbb{N}^\ell$ with $\Q\ge f(\N)$. Now it is not hard to observe (see the proof of Proposition \ref{pro neccesary condition for loop with entrance} below) that 
\begin{equation}\label{eq triple}
\displaystyle{\sum_{i=1}^{p}}\left(\displaystyle{\sum_{\alpha\in (v_i,f(\N))\overline{\Lambda}^{\Q-f(\N)}}}\overline{s}(\alpha)\right)=\displaystyle{\sum_{l=1}^{r}} (u_l,\Q)=\displaystyle{\sum_{j=1}^{q}}\left(\displaystyle{\sum_{\beta\in (w_j,f(\N))\overline{\Lambda}^{\Q-f(\N)}}}\overline{s}(\beta)\right)
\end{equation}
in $\mathbb{F}_{\overline{\Lambda}}$. We choose any $\bT\in \mathbb{N}^k$ such that $f(\bT)=\Q-f(\N)$. Now for any $l\in \{1,2,\ldots,r\}$, using (\ref{eq triple}), we can say that there exists $i\in \{1,2,\ldots,p\}$ and $(\lambda,f(\N))\in (v_i,f(\N))\overline{\Lambda}^{\Q-f(\N)}$ such that $(u_l,\Q)=\overline{s}((\lambda,f(\N)))$. Then $u_l=s(\lambda)$, $v_i=r(\lambda)$ and $\Q=d(\lambda)+f(\N)$. Therefore, $((\lambda,\bT),\N)\in (v_i,\N)\overline{f^*(\Lambda)}^\bT$ and $(u_l,\N+\bT)=\overline{s}(((\lambda,\bT),\N))$. On the other hand, for any $i\in \{1,2,\ldots,p\}$ and $((\mu,t),\N)\in (v_i,\N)\overline{f^*(\Lambda)}^\bT$, $r(\mu)=r((\mu,\bT))=v_i$ and $d(\mu)=f(\bT)=\Q-f(\N)$. Hence, $(\mu,f(\N))\in (v_i,f(\N))\overline{\Lambda}^{f(\bT)}$ and consequently, in view of (\ref{eq triple}), there exists $l\in \{1,2,\ldots,r\}$, such that $\overline{s}((\mu,f(\N)))=(u_l,\Q)$. This implies $s(\mu)=u_l$, whence $\overline{s}((\mu,\bT),\N)=(s(\mu),\N+\bT)=(u_l,\N+\bT)$. These observations show that 
\begin{equation}\label{eq ref1}
\displaystyle{\sum_{l=1}^{r}}(u_l,\N+\bT)=\displaystyle{\sum_{i=1}^{p}}\left(\displaystyle{\sum_{\alpha\in (v_i,\N)\overline{f^*(\Lambda)}^\bT}} \overline{s}(\alpha)\right).
\end{equation}
in $\mathbb{F}_{\overline{f^*(\Lambda)}}$. Similarly, using the second equality in (\ref{eq triple}), it can be shown that
\begin{equation}\label{eq ref2}
\displaystyle{\sum_{l=1}^{r}}(u_l,\N+\bT)=\displaystyle{\sum_{j=1}^{q}}\left(\displaystyle{\sum_{\beta\in (w_j,\N)\overline{f^*(\Lambda)}^\bT}} \overline{s}(\beta)\right).
\end{equation}
Now, using (\ref{eq ref1}) and (\ref{eq ref2}), we have \[\displaystyle{\sum_{i=1}^{p}} v_i(\N)=\displaystyle{\sum_{i=1}^{p}}\left(\displaystyle{\sum_{\alpha\in v_i f^*(\Lambda)^\bT}}s(\alpha)(\N+\bT)\right)=\displaystyle{\sum_{l=1}^{r}}u_l(\N+\bT)=\displaystyle{\sum_{j=1}^{q}}\left(\displaystyle{\sum_{\beta\in w_j f^*(\Lambda)^\bT}}s(\beta)(\N+\bT)\right)=\displaystyle{\sum_{j=1}^{q}} w_j(\N).\] Hence, $\rho$ is injective. 
\end{proof}

\section{Loops and generalized cycles through the lens of the talented monoid}\label{sec loop classification}

Let us start this section with the following motivating example.
\begin{example}\label{ex motivating example}
Consider the following colored directed graph:
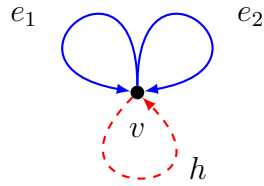
\[
\begin{tikzpicture}[scale=1]
\node[circle,draw,fill=black,inner sep=0.5pt] (p11) at (0, 0) {$.$} 

edge[-latex, blue,thick,loop, in=10, out=90, looseness=8, distance=2cm] (p11)
edge[-latex, blue,thick, loop, out=90, in=170, looseness=8, distance=2cm] (p11)

edge[-latex, red,thick, loop, dashed, out=225, in=-45, looseness=8, distance=2cm] (p11);
\node at (0, -0.5) {$v$};

\node at (-1.5,1) {$e_1$};
\node at (1.5,1) {$e_2$};
\node at (0.8,-1) {$h$};

\end{tikzpicture}
\]
We can put two different factorizations for the bicolored paths of this $2$-colored graph and consequently, obtain two non-isomorphic $2$-graphs. However, as illustrated in \cite[Example 7.1]{HMPS2}, their talented monoids coincide and both are isomorphic to $\mathbb{N}[1/2]$. 

Let us consider the $2$-graph $\Lambda$ with the above skeleton and with any one of the two factorizations. Then one can observe the following.

\begin{itemize}
\item[$(i)$] In $\Lambda$, the loop $e_1$ having degree $(1,0)$ has an entrance, namely the loop $e_2$. In $T_\Lambda$, we have $v=v(0)=2v((1,0))$ and so $^{(1,0)}v<v$. 

\item[$(ii)$] The loop $h$ having degree $(0,1)$ has no entrance. In $T_\Lambda$, the element $v=v(0)$ is periodic: $^{(0,1)}v=v$.
\end{itemize}
\end{example}

The above toy example gives a hint that the talented monoid may be effective to detect the existence of loops with and without entrance. To prove this is the main goal of this section (see Theorem \ref{th characterizing loops via TM}). Before going to accomplish this, we make the following remark, again based on the above example, to exhibit some of the anomalies one may encounter in the relationship between the geometries of a general $k$-graph and the properties of its talented monoid, which do not arise in the level of $1$-graphs.
\begin{rmk}
$(i)$ Suppose $E$ is a row-finite directed graph and $T_E$ is its talented monoid. For $a\in T_E$, it cannot happen simultaneously that $^n a=a$ and $^m a< a$ for some $n,m\in \mathbb{N}\setminus \{0\}$ and $n\neq m$. The justification is as follows: if $n>m$, then $^m a<a$ implies $^n a<~^{n-m} a$. But since $^n a=a$, this will give $a<~^{n-m}a$ which is not possible by \cite[Lemma 4.1]{hazli}. On the other hand, if $n<m$, then writing $m=nq+r$ for some $q\in \mathbb{N}\setminus \{0\}$ and $0\le r<n$, we get $^m a=^r a$ and hence $^r a<a$ for some $r<n$, which is not possible as explained in the preceding argument. Note that in the talented monoid $T_\Lambda$, where $\Lambda$ is the $2$-graph of Example \ref{ex motivating example}, we have $^{(0,1)}v=v$ and $^{(1,0)}v<v$. This weird dual behavior of $v$ can be justified in the light of Proposition \ref{pro TM of product hrg}. Observe that if the factorization of $\Lambda$ is given as $e_1 h=he_1$ and $e_2 h=h e_2$, then $\Lambda$ is isomorphic to the categorical product of the $1$-graphs $R_2^*$ and $R_1^*$. If we let $u_1$ and $u_2$ to be the sole vertices of $R_1$ and $R_2$, respectively, then under this isomorphism $v$ corresponds to $(u_2,u_1)$. Now by Proposition \ref{pro TM of product hrg}, $T_\Lambda\cong T_{R_2}\otimes T_{R_1}$ as $\mathbb{Z}^2$-monoids. Note that $^{(0,1)}(u_2\otimes u_1)=(u_2(0)\otimes u_1(1))=(u_2\otimes u_1)$ and hence, $^{(0,1)}v=v$. Again, $^{(1,0)}(u_2\otimes u_1)=(u_2(1)\otimes u_1(0))<(u_2(1)\otimes u_1(0))+(u_2(1)\otimes u_1(0))=((u_2(1)+u_2(1),u_1(0))=(u_2\otimes u_1)$ which implies $^{(1,0)}v<v$.  

$(ii)$ For any row-finite directed graph $E$, \cite[Lemma 5.6]{hazli} says that the cycles without exists correspond to minimal periodic elements in $T_E$ and vice-versa. For general $k$-graphs, although minimal periodic element in the talented monoid gives loop without entrance the converse is not true. In the $2$-graph $\Lambda$ of Example \ref{ex motivating example}, the loop $h$ has no entrance. But there are no minimal elements in $T_\Lambda$ since $v(\N+\mathbf{e}_1)<v(\N)$ for all $\N\in \mathbb{Z}^2$.  
\end{rmk}

\subsection{Loops with and without entrances}\label{ssec loops}
This subsection offers a complete classification of different types of loops (with and without sources) in a higher-rank graph in terms of the corresponding talented monoid. The results obtained here will be used in the coming sections.

We start by proving the following proposition which gives a necessary condition for every loop in a $k$-graph to have an entrance.
\begin{prop}\label{pro sufficient condition for loop with entrance}
Let $\Lambda$ be a row-finite $k$-graph without sources. Suppose for each $0\neq a\in T_\Lambda$ and $\N\in \mathbb{Z}_{<0}^k$, $^\N a\neq a$. Then every nontrivial loop in $\Lambda$ has an entrance.
\end{prop}
\begin{proof}
For the sake of a contradiction, assume that there is a nontrivial loop $\mu$ with no entrance. Let $\lambda\in s(\mu)\Lambda^{d(\mu)}$. Since $\mu$ has no entrance and $d(\mu)=d(\lambda)$, $\mu=\mu(0,d(\mu))=\mu(0,d(\lambda))=\lambda$. Therefore, $s(\mu)\Lambda^{d(\mu)}=\{\mu\}$, whence \[s(\mu)=\displaystyle{\sum_{\lambda\in s(\mu)\Lambda^{d(\mu)}}} s(\lambda)(d(\mu))=s(\mu)(d(\mu)).\] in $T_\Lambda$. This implies $^{-d(\mu)}s(\mu)=s(\mu)$, a contradiction to our assumption since $-d(\mu)\in \mathbb{Z}_{<0}^k$. 
\end{proof}

\begin{prop}\label{pro neccesary condition for loop with entrance}
Let $\Lambda$ be a row-finite $k$-graph without sources. Suppose every nontrivial loop in $\Lambda$ has an entrance. Then $^\N a\neq a$ for all $0\neq a\in T_\Lambda$ and $\N\in \mathbb{Z}_{<0}^k$.     
\end{prop}
\begin{proof}
If possible, suppose there is a nonzero element $a\in T_\Lambda$ and $\N\in \mathbb{Z}_{<0}^k$ such that $^\N a=a$. Since $\Lambda$ has no sources, we can always write $a=\displaystyle{\sum_{t=1}^{l}} v_t(\M)$, for some $\M\in \mathbb{N}^k$. Then $^{-\M}a=\displaystyle{\sum_{t=1}^{l}} v_t$. Let $b=\displaystyle{\sum_{t=1}^{l}} v_t$. Then $^\N b=^\N(^{-\M}a)=^{-\M}(^\N a)=b$. Passing this equality to $M_{\overline{\Lambda}}$ (using the isomorphism of \cite[Theorem 3.6 $(ii)$]{HMPS1}), we have \[\displaystyle{\sum_{t=1}^{l}} (v_t,\N)=\displaystyle{\sum_{t=1}^{l}} (v_t,0)\] in $M_{\overline{\Lambda}}$. By applying the confluence lemma (\cite[Lemma 3.5]{HMPS1}), we can say that there exists $\gamma\in \mathbb{F}_{\overline{\Lambda}} \setminus \{0\}$ such that \[\displaystyle{\sum_{t=1}^{l}} (v_t,\N)\longrightarrow \gamma,~~\displaystyle{\sum_{t=1}^{l}} (v_t,0)\longrightarrow \gamma.\] Again using the fact that $\Lambda$ has no sources, we can write $\gamma=\displaystyle{\sum_{j=1}^{p}} (u_j,\mathbf{q})$ for vertices $u_j\in \Lambda^0$ and some $\mathbf{q}\in \mathbb{N}^k$. Now, from the definition of the relation $\longrightarrow$ and the fact that $\overline{\Lambda}$ repeats the same pattern in each state, we have \
\begin{equation}\label{eq transition}
\displaystyle{\sum_{t=1}^{l}} (v_t,\N)\longrightarrow \displaystyle{\sum_{j=1}^{p}} (u_j,\mathbf{q}+\N).    
\end{equation}
We claim that \[\displaystyle{\sum_{j=1}^{p}} (u_j,\mathbf{q}+\N)=\displaystyle{\sum_{t=1}^{l}}\left(\displaystyle{\sum_{\alpha\in (v_t,\N)\overline{\Lambda}^\mathbf{q}}} \overline{s}(\alpha)\right).\] From the definition of $\longrightarrow$, it is clear that each $(u_j,\mathbf{q}+\N)$ is the source of some $\alpha\in (v_t,\N)\overline{\Lambda}^\mathbf{q}$ for some $t\in \{1,2,\ldots,l\}$. Fix any $t\in \{1,2,\ldots,l\}$ and choose $\alpha\in (v_t,\N)\overline{\Lambda}^\mathbf{q}$. Suppose in the transition $(\ref{eq transition})$, $(v_t,\N)$ first flow using the relation $\longrightarrow_{i_1}$. Then \[(v_t,\N)\longrightarrow_{i_1} \overline{s}(\alpha(0,\mathbf{e}_{i_1}))+x_1\] for some $x_1\in \mathbb{F}_{\overline{\Lambda}}$. If $\mathbf{q}=\mathbf{e}_{i_1}$, then we are done. Otherwise, we follow the flow of $\overline{s}(\alpha(0,\mathbf{e}_{i_1}))$ in the transition $(\ref{eq transition})$. Suppose it has flown in the direction of $\mathbf{e}_{i_2}$. Then \[(v_t,\N)\longrightarrow_{i_1} \overline{s}(\alpha(0,\mathbf{e}_{i_1}))+x_1\longrightarrow_{i_2} \overline{s}(\alpha(0,\mathbf{e}_{i_1}+\mathbf{e}_{i_2}))+x_2\] for some $x_2\in \mathbb{F}_{\overline{\Lambda}}$. Again, if $\mathbf{q}=\mathbf{e}_{i_1}+\mathbf{e}_{i_2}$, then we are done. If not, we continue the process of following the next flow. Since the state of each vertex in the final position is the same ($\mathbf{q}+\N$), every vertex in the intermediate state will flow in some direction until the transformed state becomes $\mathbf{q}+\N$. Thus, proceeding in this way, after the complete transition, we will have \[(v_t,\N)\longrightarrow_{i_1} \overline{s}(\alpha(0,\mathbf{e}_{i_1}))+x_1\longrightarrow_{i_2} \overline{s}(\alpha(0,\mathbf{e}_{i_1}+\mathbf{e}_{i_2}))+x_2\longrightarrow_{i_3}\cdots \longrightarrow_{i_m} \overline{s}(\alpha)+x_m.\] Thus, $\overline{s}(\alpha)$ is one of the $(u_j,\mathbf{q}+n)$. Since, $\displaystyle{\sum_{t=1}^{l}} (v_t,\N)\longrightarrow \gamma$, the same argument gives \[\displaystyle{\sum_{j=1}^{p}} (u_j,\mathbf{q})=\displaystyle{\sum_{t=1}^{l}}\left(\displaystyle{\sum_{\beta\in (v_t,\N)\overline{\Lambda}^{\mathbf{q}-\N}}} \overline{s}(\beta)\right).\] But it is evident that \[\displaystyle{\sum_{t=1}^{l}}\left(\displaystyle{\sum_{\alpha\in (v_t,\N)\overline{\Lambda}^\mathbf{q}}} \overline{s}(\alpha)\right)\longrightarrow \displaystyle{\sum_{t=1}^{l}}\left(\displaystyle{\sum_{\beta\in (v_t,\N)\overline{\Lambda}^{\mathbf{q}-\N}}} \overline{s}(\beta)\right)\] since $\mathbf{q}-\N> \mathbf{q}$, and for any $t\in \{1,2,\ldots,l\}$ and any path $\gamma\in (v_t,\N)\overline{\Lambda}^{\mathbf{q}-\N}$, $\gamma(0,\mathbf{q})\in (v_t,\N)\overline{\Lambda}^\mathbf{q}$. Therefore, 
\begin{equation}\label{eq transame}
\displaystyle{\sum_{j=1}^{p}} (u_j,\mathbf{q}+\N)\longrightarrow \displaystyle{\sum_{j=1}^{p}} (u_j,\mathbf{q}).    
\end{equation}
Since, the number of summand remains the same in the transition $(\ref{eq transame})$, it follows that each vertex in the list (possibly multiset) $A:=\{(u_j,\mathbf{q}+\N)~|~j=1,2,\ldots,p
\}$ is reached (via a path of degree $-\N$) from a \emph{unique} vertex in the list (possibly multiset) $B:=\{(u_j,\mathbf{q})~|~j=1,2,\ldots,p\}$. So we have a permutation $\pi\in S_l$ such that $(u_{\pi(j)},\mathbf{q}+\N)\longrightarrow (u_j,\mathbf{q})$ for all $j=1,2,\ldots,p$ or more precisely, for each $j\in \{1,2,\ldots,p\}$, there is exactly one path in $(u_{\pi(j)},\mathbf{q}+\N)\overline{\Lambda}^{-\N}$ say, $(\lambda_j,\mathbf{q}+\N)$ and $\overline{s}(\lambda_j,\mathbf{q}+\N)=(u_j,\mathbf{q})$. Let $o(\pi)=\nu$. Since $\overline{\Lambda}$ repeats the pattern in each state, we have a unique path \[(\lambda_{\pi^{\nu-1}(j)},\mathbf{q}+\nu \N)(\lambda_{\pi^{\nu-2}(j)},\mathbf{q}+(\nu-1) \N)\cdots (\lambda_{\pi(j)},\mathbf{q}+2 \N)(\lambda_{j},\mathbf{q}+\N)\] of degree $-\nu \N$ with range $(u_{\pi^\nu(j)},\mathbf{q}+\nu\N)=(u_j,\mathbf{q}+\nu \N)$ and the source of this path is $(u_j,\mathbf{q})$. Thus, $u_j(\mathbf{q}+\nu \N)=u_j(\mathbf{q})$ in $T_{\Lambda}$ and there is a nontrivial loop $\mu_j\in \Lambda^{-\nu \N}$ with $s(\mu_j)=r(\mu_j)=u_j$, namely \[\mu_j=\lambda_{\pi^{\nu-1}(j)}\lambda_{\pi^{\nu-2}(j)}\cdots \lambda_{\pi(j)}\lambda_{j}.\] If this loop has an entrance $\lambda$, then for any $\tau\in s(\lambda)\Lambda^{-\nu \N-d(\lambda)}$, since $\mu_j(0,d(\lambda))\neq \lambda=(\lambda\tau)(0,d(\lambda))$, $\mu_j\neq \lambda\tau$. Consequently, in $T_\Lambda$, \[u_j(\mathbf{q}+\nu \N)=\displaystyle{\sum_{\alpha\in u_j\Lambda^{-\nu \N}}} s(\alpha)(\mathbf{q})\ge u_j(\mathbf{q})+s(\tau)(\mathbf{q})> u_j(\mathbf{q}).\] But it is not possible as we have $u_j(\mathbf{q}+\nu \N)=u_j(\mathbf{q})$. Therefore, the loop $\mu_j$ has no entrance which contradicts the hypothesis. Hence, $^\N a\neq a$ for all $0\neq a\in T_\Lambda$ and $\N\in \mathbb{Z}_{<0}^k$. 
\end{proof}

We classify the existence of different kinds of loops (with and without entrance) in a $k$-graph via the talented monoid.
\begin{thm}\label{th characterizing loops via TM}
\textup{($cf.$ \cite[Proposition 4.2]{hazli})} Let $\Lambda$ be a row-finite $k$-graph with no sources. Then

$(i)$ $\Lambda$ has a nontrivial loop without an entrance if and only if there exist $0\neq a\in T_\Lambda$ and $\N\in \mathbb{Z}_{<0}^k$ such that $^\N a=a$.

$(ii)$ $\Lambda$ has a loop with an entrance if and only if there is an element $a\in T_\Lambda$ and $\N\in \mathbb{N}^k$ such that $^\N a< a$. 

$(iii)$ $\Lambda$ has a nontrivial loop if and only if there is an element $0\neq a\in T_\Lambda$ and $\N\in \mathbb{N}^k\setminus \{0\}$ such that $^\N a\le a$. 
\end{thm}

\begin{proof}
$(i)$ Follows immediately from Propositions \ref{pro sufficient condition for loop with entrance} and \ref{pro neccesary condition for loop with entrance}.

$(ii)$ Suppose $\mu$ is a loop in $\Lambda$ with an entrance $\lambda$. Since $\Lambda$ has no sources and $d(\mu)\ge d(\lambda)$, we can choose $\tau\in s(\lambda)\Lambda^{d(\mu)-d(\lambda)}$. Then $\mu\neq \lambda\tau$ otherwise $\mu(0,d(\lambda))=(\lambda\tau)(0,d(\lambda))=\lambda$ contradicting the fact that $\lambda$ is an entrance of $\mu$. This implies that\[s(\mu)=\displaystyle{\sum_{\alpha\in s(\mu)\Lambda^{d(\mu)}}} s(\alpha)(d(\mu))\ge s(\mu)(d(\mu))+s(\lambda\tau)(d(\mu))>s(\mu)(d(\mu))\] and hence, $^{d(\mu)}s(\mu) < s(\mu)$.

Conversely, assume that $^\N a< a$ for some $a\in T_\Lambda$ and $\N\in \mathbb{N}^k$. Obviously, $a\neq 0$ and $\N\neq 0$. Using the fact that $\Lambda$ has no sources, we can write $a$ as $a=\displaystyle{\sum_{t=1}^{l}}~ v_t(\M)$ for some $m\in \mathbb{N}^k$. Now since $a> ~^\N a$, there exists $b\in T_\Lambda\setminus \{0\}$ such that $a=~^\N a+b$. We now bypass this equality to the $k$-graph monoid $M_{\overline{\Lambda}}$ and subsequently, obtain \[\displaystyle{\sum_{t=1}^{l}} (v_t,\M)=\displaystyle{\sum_{t=1}^{l}} (v_t,\M+\N)+\phi(b),\] where $\phi:T_\Lambda\longrightarrow M_{\overline{\Lambda}}$ is the isomorphism of \cite[Theorem 3.16]{HMPS1}. Now, applying \cite[Lemma 3.5]{HMPS1} followed by \cite[Lemma 3.1]{HMPS1}, we can find $\gamma,\alpha,\beta\in \mathbb{F}_{\overline{\Lambda}}\setminus \{0\}$ such that $\gamma=\alpha+\beta$, $\displaystyle{\sum_{t=1}^{l}} (v_t,\M)\longrightarrow \gamma$, $\displaystyle{\sum_{t=1}^{l}} (v_t,\M+\N)\longrightarrow \alpha$ and $\phi(b)\longrightarrow \beta$. Again, since $\overline{\Lambda}$ has no sources, we can choose $\gamma$ to be such that all the vertices in its support lie in a common state $\mathbf{q}$, where $\mathbf{q}> \M+\N$. Then $\alpha$ being a summand of $\gamma$ in $\mathbb{F}_{\overline{\Lambda}}$, must be of the form $\alpha=\displaystyle{\sum_{j=1}^{p}}(u_j,\mathbf{q})$ for some vertices $u_j\in \Lambda^0$. Thus, we have \[\displaystyle{\sum_{t=1}^{l}} (v_t,\M+\N)\longrightarrow \displaystyle{\sum_{j=1}^{p}} (u_j,\mathbf{q}).\] Since the skew-product $k$-graph $\overline{\Lambda}$ repeats the same pattern in each state, this implies \[\displaystyle{\sum_{t=1}^{l}}(v_t,\M)\longrightarrow \displaystyle{\sum_{j=1}^{p}} (u_j,\mathbf{q}-\N).\] Also, \[\displaystyle{\sum_{t=1}^{l}} (v_t,\M)\longrightarrow \gamma=\displaystyle{\sum_{j=1}^{p}} (u_j,\mathbf{q})+\beta.\] Now, a similar argument as used in the proof of Proposition \ref{pro neccesary condition for loop with entrance}, yields \[\displaystyle{\sum_{j=1}^{p}} (u_j,\mathbf{q}-\N)=\displaystyle{\sum_{t=1}^{l}}\left(\displaystyle{\sum_{\lambda\in (v_t,\M)\overline{\Lambda}^{\mathbf{q}-\N-\M}}}\overline{s}(\lambda)\right)\] and \[\gamma=\displaystyle{\sum_{t=1}^{l}}\left(\displaystyle{\sum_{\mu\in (v_t,\M)\overline{\Lambda}^{\mathbf{q}-\M}}}\overline{s}(\mu)\right).\] Since $\mathbf{q}-\N-\M<\mathbf{q}-\M$, this implies 
\begin{equation}\label{eq splitflow}
\displaystyle{\sum_{j=1}^{p}} (u_j,\mathbf{q}-\N)\longrightarrow \displaystyle{\sum_{j=1}^{p}} (u_j,\mathbf{q})+\beta.    
\end{equation}

Since $\beta\neq 0$, the number of summand in the right hand side is greater than the number of summand in the left hand side of the transition (\ref{eq splitflow}). Then there should exist some $i\in \{1,2,\ldots,k\}$ with $\mathbf{e}_i\le \N$ such that some vertex appearing in an intermediate step of (\ref{eq splitflow}) receives at least two different paths of degree $\mathbf{e}_i$. Now, following the same line of argument used in the proof of \cite[Proposition 4.2 $(ii)$]{hazli}, we can obtain a nonempty subset $A=\{j_1,j_2,\ldots,j_q\}$ of $\{1,2,\ldots,p\}$ and a permutation $\pi\in S_q$ such that for all $i=1,2,\ldots,q$, there is exactly one path $(\lambda_{j_i},\mathbf{q}-\N)\in \overline{\Lambda}^{\N}$ with $\overline{r}(\lambda_{j_i},\mathbf{q}-\N)=(u_{j_{\pi(i)}},\mathbf{q}-\N)$ and $\overline{s}(\lambda_{j_i},\mathbf{q}-\N)=(u_{j_i},\mathbf{q})$. Moreover, there exist $\ell\in \{1,2,\ldots,q\}$ and $0\le \mathbf{r}<\N$ such that $|(\lambda_{j_\ell},\mathbf{q}-\N)(\mathbf{r})\overline{\Lambda}^{\mathbf{e}_i}|\ge 2$ for some $i\in \{1,2,\ldots,k\}$ with $\mathbf{e}_i\le \N-\mathbf{r}$. Let $o(\pi)=\nu$. Then we have a path \[(\lambda_{j_{\pi^{\nu-1}(\ell)}},\mathbf{q}-\nu \N)(\lambda_{j_{\pi^{\nu-2}(\ell)}},\mathbf{q}-(\nu-1) \N)\cdots (\lambda_{j_{\pi(\ell)}},\mathbf{q}-2 \N)(\lambda_{j_\ell},\mathbf{q}-\N)\] in $\overline{\Lambda}$ of degree $\nu \N$ with range $(u_{j_{\pi^\nu(\ell)}},\mathbf{q}-\nu\N)=(u_{j_\ell},\mathbf{q}-\nu \N)$ and source $(u_{j_\ell},\mathbf{q})$. This gives a loop $\mu\in \Lambda^{\nu \N}$ with $s(\mu)=r(\mu)=u_{j_\ell}$, namely \[\mu=\lambda_{j_{\pi^{\nu-1}(\ell)}}\lambda_{j_{\pi^{\nu-2}(\ell)}}\cdots \lambda_{j_{\pi(\ell)}}\lambda_{j_\ell}.\] We now show that $\mu$ has an entrance. Since $|(\lambda_{j_\ell},\mathbf{q}-\N)(\mathbf{r})\overline{\Lambda}^{\mathbf{e}_i}|\ge 2$, $|\mu((\nu-1)\N+\mathbf{r})\Lambda^{\mathbf{e}_i}|\ge 2$ and so we can choose \[\lambda\in \mu((\nu-1)\N+\mathbf{r})\Lambda^{\mathbf{e}_i}\setminus \{\mu((\nu-1)\N+\mathbf{r},(\nu-1)\N+\mathbf{r}+\mathbf{e}_i)\}.\] Let $\alpha=\mu(0,(\nu-1)\N+\mathbf{r})\lambda$. Since $\mathbf{e}_i\le \N-\mathbf{r}$, $d(\alpha)=(\nu-1)\N+\mathbf{r}+\mathbf{e}_i\le \nu\N=d(\mu)$. Also, $\mu(0,(\nu-1)\N+\mathbf{r}+\mathbf{e}_i)\neq \alpha$. Therefore, $\alpha$ is an entrance to the loop $\mu$. This finishes the proof.

$(iii)$ Follows directly from $(i)$ and $(ii)$. 
\end{proof}
To state our next result, we formally define the higher-rank analogues of Condition (L) and Condition (NE) of directed graphs. 
\begin{dfn}\label{def condition L and NE in higher-dimension}
Suppose $\Lambda$ is a higher-rank graph. We say that $\Lambda$ satisfies 

$(i)$ \emph{Condition} (HL), if every nontrivial loop in $\Lambda$ has an entrance.

$(ii)$ \emph{Condition} (HNE), if no loop in $\Lambda$ has an entrance. 
\end{dfn}
In view of the results we have so far, we can now obtain the following corollary which characterizes the above defined conditions on a $k$-graph via its associated talented monoid.
\begin{cor}\label{cor characterizing HL and HNE via TM}
Let $\Lambda$ be a row-finite $k$-graph without sources. Then

$(i)$ $\Lambda$ has Condition \textup{(HL)} if and only if $^\N a\neq a$ for all $0\neq a\in T_\Lambda$ and $\N\in \mathbb{Z}_{<0}^k$,

$(ii)$ $\Lambda$ has Condition \textup{(HNE)} if and only if there are no $a\in T_\Lambda$ and $\mathbf{q}\in \mathbb{N}^k$ such that $^\mathbf{q} a< a$.

Furthermore, if $|\Lambda^0|<\infty$, then $\Lambda$ has Condition \textup{(HNE)}  if and only if for any $0\neq a\in T_\Lambda$ and $i\in \{1,2,...,k\}$, there exists $n_i\in \mathbb{N}\setminus \{0\}$ such that $^{n_i \mathbf{e}_i} a=a$.
\end{cor}
\begin{proof}
$(i)$ follows from Propositions \ref{pro sufficient condition for loop with entrance} and \ref{pro neccesary condition for loop with entrance} whereas $(ii)$ follows from Theorem \ref{th characterizing loops via TM} $(ii)$. We now move towards the final statement. First let us show that under the said condition of the talented monoid, $\Lambda$ has Condition (HNE). Choose any $0\neq a\in T_\Lambda$ and $\mathbf{q}=(q_1,q_2,\ldots,q_k)\in \mathbb{N}^k$. By part $(ii)$ of this Corollary, it suffices to show that $^\mathbf{q} a\nless a$. If possible suppose $^\mathbf{q}a<a$. Now for any $i\in \{1,2,\ldots,k\}$, there exists a positive integer $n_i$ such that $^{n_i\mathbf{e}_i}a=a$. Let $l=l.c.m(n_1,n_2,\ldots,n_k)$. Then $l>0$ and hence $^{l\mathbf{q}}a<a$. But for all $i=1,2,\ldots,k$, $^{lq_i\mathbf{e}_i}a=~^{(\frac{l q_i}{n_i})n_i\mathbf{e}_i}a=a$, and hence \[^{l\mathbf{q}}a=~^{lq_1\mathbf{e}_1+lq_2\mathbf{e}_2+\cdots+lq_k\mathbf{e}_k}a=a,\] a contradiction. Therefore, $^\mathbf{q}a\nless a$. This shows that $\Lambda$ has Condition (HNE). Now, we prove the other way round. So assume that $\Lambda$ has Condition (HNE). Suppose $0\neq a\in T_\Lambda$ and $i\in \{1,2,\ldots,k\}$. We can write $a=\displaystyle{\sum_{t=1}^{p}} v_t(\M)$ for some vertices $v_t$ of $\Lambda$ and some $\M\in \mathbb{Z}^k$. Let us choose any $t\in \{1,2,\ldots,p\}$. Using the fact that $\Lambda$ has no sources, we can transform $(v_t,0)\in \overline{\Lambda}^0$ in $\mathbb{F}_{\overline{\Lambda}}$ using the relation $\longrightarrow_i$ and continue doing this for all its ancestors until we reach a stage where each ancestor is the source of some loop entirely made of edges of degree $\mathbf{e}_i$. Since $\Lambda^0$ is finite, such a stage will definitely arise. Moreover, since $\Lambda$ has Condition (HNE), the loops, obtained eventually, have no entrances. Transiting everything in $T_\Lambda$ and following the proof of Proposition \ref{pro sufficient condition for loop with entrance}, we can write $v_t=\displaystyle{\sum_{j=1}^{m_t}}w_j^t(\ell_j\mathbf{e}_i)$, where $\ell_j\in \mathbb{N}$ and $w_j^t\in \Lambda^0$ (not necessarily distinct) are such that for each $j$, there exists a positive integer $q_{j,t}$ with $^{q_{j,t}\mathbf{e}_i}w_j^t=w_j^t$. Let $q_t:=l.c.m(q_{1,t},q_{2,t},\ldots,q_{{m_t},t})$. Then \[^{q_t\mathbf{e}_i}v_t=\displaystyle{\sum_{j=1}^{m_t}}~^{q_t\mathbf{e}_i}w_j^t(\ell_j\mathbf{e}_i)=\displaystyle{\sum_{j=1}^{m_t}}w_j^t(\ell_j\mathbf{e}_i)=v_t.\] Finally, setting $n_i:=l.c.m(q_1,q_2,\ldots,q_p)$, we have \[^{n_i\mathbf{e}_i}a=\displaystyle{\sum_{t=1}^{p}}~^{n_i\mathbf{e}_i}v_t(\M)=\displaystyle{\sum_{t=1}^{p}}v_t(\M)=a.\]  
\end{proof}

As a simple application of Corollary \ref{cor characterizing HL and HNE via TM}, we now obtain the following interesting result regarding the preservation of Conditions (HL) and (HNE) under pullback. 
\begin{cor}\label{cor Condition HL and HNE under pullback}
Let $\Lambda$ be a row-finite $\ell$-graph without sources and $f:\mathbb{N}^k\longrightarrow \mathbb{N}^\ell$ a surjective monoid homomorphism. Then the following hold.

$(i)$ $\Lambda$ has Condition \textup{(HNE)} if and only if $f^*(\Lambda)$ has Condition \textup{(HNE)}. 

$(ii)$ If $f^*(\Lambda)$ has Condition \textup{(HL)}, then $\Lambda$ has Condition \textup{(HL)}. The converse holds if $f$ is $0$-restricted, i.e., $f^{-1}(\{0\})=\{0\}$. 
\end{cor}
\begin{proof}
$(i)$ Suppose $\Lambda$ has Condition (HNE). If there exist $a\in T_{f^*(\Lambda)}$ and $\Q\in \mathbb{N}^k$ such that $^\Q a< a$, then applying the isomorphism $\rho:T_{f^*(\Lambda)}\longrightarrow T_\Lambda$ of Proposition \ref{pro TM under pullbacks}, we have $\rho(^\Q a)<\rho(a)$. This implies $^{f(\Q)}\rho(a)<\rho(a)$ in $T_\Lambda$. But this is not possible in view of Corollary \ref{cor characterizing HL and HNE via TM}$(ii)$. Hence, there are no $a\in T_{f^*(\lambda)}$ and $\Q\in \mathbb{N}^k$ such that $^\Q a<a$. Corollary \ref{cor characterizing HL and HNE via TM}$(ii)$ now shows that $f^*(\Lambda)$ has (HNE). The other way round is similar.

$(ii)$ Suppose $f^*(\Lambda)$ has Condition (HL). Let $0\neq a\in T_\Lambda$ and $\N\in \mathbb{Z}_{<0}^\ell$. Choose any $\M\in \mathbb{N}^k$ such that $f(\M)=\N$ and let $b=\rho^{-1}(a)$. Then $b\neq 0$ and $\M\in \mathbb{Z}_{<0}^k$ Now if $^\N a=a$, then $\rho(^\M b)=\rho(b)$ and consequently, $^\M b=b$ in $T_{f^*(\Lambda)}$. This is a contradiction in view of Corollary \ref{cor characterizing HL and HNE via TM}$(i)$. Thus, $^\N a\neq a$ for all $0\neq a\in T_\Lambda$ and $\N\in \mathbb{Z}_{<0}^\ell$. Then $\Lambda$ has Condition (HL) by Corollary \ref{cor characterizing HL and HNE via TM}$(i)$. Now, assume that $f$ is $0$-restricted and $\Lambda$ has Condition (HL). If there exist $0\neq a\in T_{f^*(\Lambda)}$ and $\N\in \mathbb{Z}_{<0}^k$ with $^\N a=a$, then applying the isomorphism $\rho$, we have $^{f(\N)}\rho(a)=\rho(a)$, where $\rho(a)\neq 0$ since $\rho$ is injective and $f(\N)\neq 0$ since $f$ is $0$-restricted. But then we have a contradiction due to Corollary \ref{cor characterizing HL and HNE via TM}$(i)$. Therefore, no such $a$ and $\N$ exist, which implies that $f^*(\Lambda)$ has Condition (HL).     
\end{proof}
\begin{rmk}\label{rem proof holds for any f}
The ``only if'' direction in Corollary \ref{cor Condition HL and HNE under pullback}$(i)$ can also be proved even if $f$ is not surjective. For $a\in T_{f^*(\Lambda)}$ and $\Q\in \mathbb{N}^k$, $^\Q a<a$ implies $a=~^\Q a+b$ for some $0\neq b\in T_{f^*(\Lambda)}$. Now it is easy to see that the homomorphism $\rho:T_{f^*(\Lambda)}\longrightarrow T_\Lambda$ (see the proof of Proposition \ref{pro TM under pullbacks}) maps nonzero elements of $T_{f^*(\Lambda)}$ to nonzero elements of $T_\Lambda$. Hence, $\rho(b)\neq 0$ in $T_\Lambda$ and consequently, $^{f(\Q)}\rho(a)<\rho(a)$ giving the same contradiction. Similarly, the converse part of Corollary \ref{cor Condition HL and HNE under pullback}$(ii)$ can also be proved for any $0$-restricted homomorphism $f$ (not necessarily surjective).     
\end{rmk}
In Corollary \ref{cor Condition HL and HNE under pullback}$(ii)$, the assumption that $f$ is $0$-restricted cannot be dropped. To understand why it is required, consider the $2$-graph $\Lambda$ of Example \ref{ex motivating example}. Note that $\Lambda\cong f^*(R_2^*)$, where $R_2$ is the \emph{rose with two petals} and $f:\mathbb{N}^2\longrightarrow \mathbb{N}$ is the homomorphism defined by $f((x,y)):=x$. Clearly, $R_2^*$ has Condition (HL) but $\Lambda$ does not since the loop $h$ has no entrance. This is mainly because $f((0,1))=0$.

\subsection{Reachability from loops with entrances, generalized cycles, and pure infiniteness via the talented monoid}\label{ssec purely infinite simple}
In this subsection, our main aim is to obtain a sufficient talented monoid criterion (see Theorem \ref{th the TM criterion for purely infinite simplicity} below) ensuring pure infinite simplicity of a Kumjian--Pask algebra. It is apparent from the literature that reachability fof a vertex from loops with entrances plays significant role to obtain purely infinite simple higher-rank graph algebras. Hence, we first characterize this via the talented monoid.

\begin{lem}\label{lem connects to a loop with entrance}
Let $\Lambda$ be a row-finite $k$-graph without sources. Let $v\in \Lambda^0$. Then $v$ is reachable from a loop with an entrance if and only if there exists $a\in T_\Lambda$ and $\N\in \mathbb{N}^k$ such that $v(0) \ge a>~ ^\N a$. 
\end{lem}
\begin{proof}
We first prove the sufficiency. Since $v\ge a$, there exists $x\in T_\Lambda$ such that $v=a+x$. Passing this equality to $M_{\overline{\Lambda}}$ and using the Confluence lemma (\cite[Lemma 3.5]{HMPS1}), we have $\alpha\in \mathbb{F}_{\overline{\Lambda}}\setminus \{0\}$ such that $(v,0)\longrightarrow \alpha$ and $a+x\longrightarrow \alpha$. For simplicity, we are using the same notation to denote elements of $T_\Lambda$ and $M_{\overline{\Lambda}}$. By \cite[Lemma 3.1]{HMPS1}, there exists $\alpha_1\in \mathbb{F}_{\overline{\Lambda}}\setminus \{0\}$ and $\alpha_2\in \mathbb{F}_{\overline{\Lambda}}$ such that $a\longrightarrow \alpha_1$ and $x\longrightarrow \alpha_2$. Note that $\alpha_2=0$ if $v=a$. Again $a=~^\N a+y$ for some $y\in T_\Lambda\setminus \{0\}$. Applying the same process, we have a $\beta\in \mathbb{F}_{\overline{\Lambda}}\setminus \{0\}$ such that $a\longrightarrow \beta$ and $^\N a+y\longrightarrow \beta$. Now, we have $a\longrightarrow \alpha_1$ and $a\longrightarrow \beta$, which implies $\alpha_1=\beta$ in $M_{\overline{\Lambda}}$. By applying the confluence lemma again, we have some $\gamma\in \mathbb{F}_{\overline{\Lambda}}\setminus \{0\}$ such that $\alpha_1\longrightarrow \gamma$ and $\beta\longrightarrow \gamma$. Then $a\longrightarrow \alpha_1\longrightarrow \gamma$ and $^\N a+y\longrightarrow \beta\longrightarrow \gamma$. Now, if we apply the exact same line of argument used in the proof of Theorem \ref{th characterizing loops via TM} $(ii)$, then we will obtain a vertex $u\in \Lambda^0$ which appears (with some state) in the support of $\gamma$ such that $u=s(\mu)=r(\mu)$ for a loop $\mu\in \Lambda$ with an entrance. Since $\alpha_1\longrightarrow \gamma$, there is a path $\delta\in \Lambda u$ which ends at some vertex $w$ in the support of $\pi(\alpha_1)$ (where $\pi:T_\Lambda(\cong M_{\overline{\Lambda}})\longrightarrow M_\Lambda$ is the forgetful homomorphism of \cite[Theorem 3.16 $(i)$]{HMPS1}). Again, $(v,0)\longrightarrow \alpha=\alpha_1+\alpha_2$ implies that there is a path $\tau\in v\Lambda w$. Finally, these two paths concatenate to give the path $\tau\delta\in v\Lambda u$ and hence, $v$ is reachable from the loop $\mu$ which has an entrance. For the reader's convenience, the transitions used in the proof are presented below.
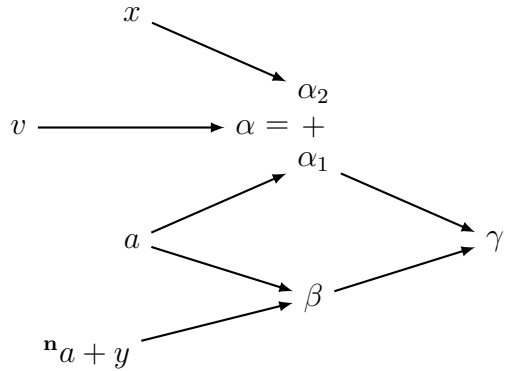
\[
\begin{tikzpicture}[scale=1.5]                                                        
\node[] (A) at (0,0) {$v$};
\node[] (B) at (2,0) {$\alpha$};
\node[] (C) at (1,-1) {$a$};
\node[] (D) at (2.6,-0.3) {$\alpha_1$};
\node[] (E) at (2.6,-1.5) {$\beta$};
\node[] (F) at (4.2,-1) {$\gamma$};
\node[] (G) at (0.6,-2) {$^\N a+y$};
\node[] (H) at (1,1) {$x$}; 
\node[] (J) at (2.6,0.3) {$\alpha_2$};

\path[->, >=latex,thick] (A) edge [] node[above=0.05cm]{} (B);
\path[->, >=latex,thick] (C) edge [] node[above=0.05cm]{} (D);
\path[->, >=latex,thick] (C) edge [] node[above=0.05cm]{} (E);
\path[->, >=latex,thick] (D) edge [] node[above=0.05cm]{} (F);
\path[->, >=latex,thick] (E) edge [] node[above=0.05cm]{} (F);
\path[->, >=latex,thick] (G) edge [] node[above=0.05cm]{} (E);
\path[->, >=latex,thick] (H) edge [] node[above=0.05cm]{} (J);

\node at (2.6,0) {$+$};
\node at (2.28,0) {$=$};

\end{tikzpicture}
\]

For the necessity, suppose $v$ is reachable from a loop $\mu$ which has an entrance $\alpha$. Then there is a path $\lambda\in v\Lambda s(\mu)$. In $T_\Lambda$, we have \[v=v(0)=\displaystyle{\sum_{\gamma\in v\Lambda^{d(\lambda)}}} s(\gamma)(d(\lambda))\ge s(\lambda)(d(\lambda))=s(\mu)(d(\lambda)).\] Again, the argument used in the first part of the proof of Theorem \ref{th characterizing loops via TM} $(ii)$, gives $^{d(\mu)}s(\mu)< s(\mu)$. Therefore, we have \[v\ge s(\mu)(d(\lambda))> ~^{d(\mu)}s(\mu)(d(\lambda)).\] Letting $a=s(\mu)(d(\lambda))$ and $\N=d(\mu)$, we are done.
\end{proof}

\begin{cor}\label{cor reachability from a loop with entrance}
Let $\Lambda$ be a row-finite $k$-graph with no sources. Then every vertex of $\Lambda$ is reachable from a loop with an entrance if and only if for any $0\neq b\in T_\Lambda$, there exist $a\in T_\Lambda$ and $\N\in \mathbb{N}^k$ such that $b\ge a>~^\N a$.     
\end{cor}
\begin{proof}
Suppose $v$ is any vertex of $\Lambda$. If the condition holds, then there is an $a\in T_\Lambda$ and $\N\in \mathbb{N}^k$ such that $v(0)\ge a>~^\N a$. By Lemma \ref{lem connects to a loop with entrance}, $v$ is then reachable from a loop with an entrance. Now assume that every vertex is reachable from a loop with an entrance and take any $0\neq b\in T_\Lambda$. Then $b\ge v(\M)$ for some $v\in \Lambda^0$ and $\M\in \mathbb{N}^k$. Since $v$ is reachable from a loop with an entrance, applying Lemma \ref{lem connects to a loop with entrance}, we have $v(0)\ge a>~^\N a$ for some $a\in T_\Lambda$ and $\N\in \mathbb{N}^k$. Then $b\ge v(\M)\ge~^\M a>~^\M(^\N a)=~^\N(^\M a)$ and we are done. 
\end{proof}

Now we have the necessary tools to obtain the following theorem giving a sufficient condition for purely infinite simple (discrete and analytic) algebras of row-finite $k$-graphs without sources.

\begin{thm}\label{th the TM criterion for purely infinite simplicity}
Let $\Lambda$ be a row-finite $k$-graph without sources. Suppose the following hold:

$(i)$ $T_\Lambda$ is $\mathbb{Z}^k$-order ideal simple,

$(ii)$ $\mathbb{Z}^k$ acts freely on $T_\Lambda$ and 

$(iii)$ for any $0\neq b\in T_\Lambda$, there exist $a\in T_\Lambda$ and $\N\in \mathbb{N}^k$ such that $b\ge a >~^\N a$. 

Then both the $C^*$-algebra $C^*(\Lambda)$ and the Kumjian--Pask algebra $\KP_\mathsf{F}(\Lambda)$ are simple and purely infinite.     
\end{thm}
\begin{proof}
The condition that $T_\Lambda$ is free of any nontrivial $\mathbb{Z}^k$-order ideal implies that $\Lambda$ is cofinal (see \cite[Proposition 4.8]{HMPS1}), whereas the second condition implies that $\Lambda$ is aperiodic by \cite[Proposition 5.2]{HMPS1}. The third condition, in view of Corollary \ref{cor reachability from a loop with entrance}, says that every vertex is reachable from a loop with an entrance. The $C^*$-algebra case follows by applying Theorem \ref{th pis of C*-algebra}. Since an entrance to a loop $\mu$ is also an entrance to the corresponding generalized cycle $(\mu,s(\mu))$, every vertex in $\Lambda$ is in fact reachable from a generalized cycle with an entrance. The case for the Kumjian--Pask algebra now follows from \cite[Theorem 5.4]{LPIS}. 
\end{proof}

We illustrate the application of the above theorem via the following example.
\begin{example}\label{ex KP is purely infinite simple via TM}
Let $\Lambda$ be the $2$-graph of \cite[Example 6.4]{HMPS1}. The $1$-skeleton is shown below.
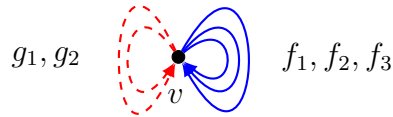
\[
\begin{tikzpicture}[scale=0.35]
\node[circle,draw,fill=black,inner sep=0.5pt] (p11) at (1, 1) {$.$} 
edge[-latex, blue, thick,loop, out=40, in=-40, min distance=80, looseness=1.5] (p11)
edge[-latex, blue,thick,loop, out=50, in=-50, min distance=123, looseness=2.2] (p11)
edge[-latex, blue,thick, loop, out=60, in=-60, min distance=190, looseness=2.3] (p11)
edge[-latex, red,thick, loop, dashed, out=125, in=235, min distance=115, looseness=1.8] (p11)
edge[-latex, red, thick, loop, dashed, out=115, in=245, min distance=190, looseness=1.7] (p11);
                                                               
\node at (0.9, -0.5) {$v$};
\node at (-4, 1) {$g_1,g_2$};
\node at (7,1) {$f_1,f_2,f_3$};
\end{tikzpicture}
\]
The factorization rules for the bicolored paths are given as follows: $f_ig_j=g_if_j$ for $i,j=1,2$ and $f_3 g_i=g_if_3$ for $i=1,2$. Let $J$ be any nonzero $\mathbb{Z}^2$ order ideal of $T_\Lambda$ and choose $0\neq a\in J$. Then $a\ge v(\N)$ for some $\N\in \mathbb{Z}^2$, which implies $v=~^{-\N}v(\N)\in J$. Consequently, $v(\M)\in J$ for all $\M\in \mathbb{Z}^2$, whence $J=T_\Lambda$. Thus, $T_\Lambda$ is $\mathbb{Z}^2$-order ideal simple. Again $\mathbb{Z}^2$ acts freely on $T_\Lambda$ (see \cite[Example 6.4]{HMPS1} for the justification). If $0\neq b\in T_\Lambda$, then since $b\ge v(\M)$ for some $\M\in \mathbb{Z}^2$ and $v=3v(\mathbf{e}_1)> v(\mathbf{e}_1)$, we have $b\ge v(\M)> ~^{\mathbf{e}_1}v(\M)$. Hence, all the three conditions of Theorem \ref{th the TM criterion for purely infinite simplicity} are satisfied. It follows that $\KP_\mathsf{F}(\Lambda)$ is simple and purely infinite.  
\end{example}

Our next theorem exhibits yet another talent of the talented monoid as it provides a necessary condition, via the talented monoid, for the existence of a generalized cycle. Moreover, the monoid can completely determine when a generalized cycle has an entrance. To prove this, we need the following lemma.
\begin{lem}\label{lem generalized cycle}
\textup{($cf.$ \cite[Lemma 3.7]{Evans})} Suppose $\Lambda$ is a row-finite $k$-graph without sources. Let $(\mu,\nu)$ be a generalized cycle in $\Lambda$. Let $v$ be the common source of $\mu$ and $\nu$. Then $^{d(\mu)}v\le~^{d(\nu)}v$ in $T_\Lambda$. Furthermore, the generalized cycle $(\mu,\nu)$ has an entrance if and only if $~^{d(\mu)}v<~^{d(\nu)}v$.    
\end{lem}
\begin{proof}
Let $\N=(d(\mu)\vee d(\nu))-d(\mu)$ and $\M=(d(\mu)\vee d(\nu))-d(\nu)$. We first claim that $\Ext(\mu,\{\nu\})=v\Lambda^\N$. The inclusion $\Ext(\mu,\{\nu\})\subseteq v\Lambda^\N$ follows from the definition. For the other way round, choose any $\lambda\in v\Lambda^\N$. Since $(\mu,\nu)$ is a generalized cycle, $\Ext(\mu,\{\nu\})$ is an exhaustive subset of $v\Lambda$ by \cite[Lemma 3.2]{Evans}. Hence, there is some $\alpha\in \Ext(\mu,\{\nu\})$ such that $\MCE(\lambda,\alpha)\neq \emptyset$. But since $d(\lambda)=d(\alpha)=\N$, we should have $\lambda=\alpha$, whence $\lambda\in \Ext(\mu,\{\nu\})$. So the claim is established. Next we define an injective map from $\Ext(\mu,\{\nu\})$ into $v\Lambda^\M$. For this note that if $\alpha\in \Ext(\mu,\{\nu\})$, then there is a unique $\beta_\alpha\in v\Lambda^\M$ such that $\mu\alpha=\nu\beta_\alpha$. Now define 
\begin{align*}
\varphi:\Ext(\mu,\{\nu\})&\longrightarrow v\Lambda^\M\\
\alpha&\longmapsto \beta_\alpha.
\end{align*}
This is clearly injective since \[\alpha_1\neq \alpha_2\Rightarrow \mu\alpha_1\neq \mu\alpha_2\Rightarrow \nu\beta_{\alpha_1}\neq \nu\beta_{\alpha_2}\Rightarrow \beta_{\alpha_1}\neq \beta_{\alpha_2}.\] Also, observe that $s(\alpha)=s(\varphi(\alpha))$ since $\mu\alpha=\nu\beta_\alpha$. Now, using the map $\varphi$ together with this observation and our earlier claim, we have 
\begin{align*}
v(-\N)=\displaystyle{\sum_{\alpha\in v\Lambda^\N}}s(\alpha)=\displaystyle{\sum_{\alpha\in \Ext(\mu,\{\nu\})}} s(\alpha)
=\displaystyle{\sum_{\alpha\in \Ext(\mu,\{\nu\})}} s(\varphi(\alpha))=\displaystyle{\sum_{\beta\in \image(\varphi)}} s(\beta)
\le \displaystyle{\sum_{\beta\in v\Lambda^\M}} s(\beta)=v(-\M)
\end{align*}
in $T_\Lambda$. Hence, \[^{d(\mu)}v=~^{d(\mu)\vee d(\nu)}v(-\N)\le~^{d(\mu)\vee d(\nu)}v(-\M)=~^{d(\nu)}v.\] 

Now, suppose $(\mu,\nu)$ has an entrance $\tau$. Choose any $\eta\in s(\tau)\Lambda^{(\M\vee d(\tau))-d(\tau)}$. If $(\tau\eta)(0,\M)\in \image(\varphi)$, then $\mu\alpha=\nu (\tau\eta)(0,\M)$ for some $\alpha\in \Ext(\mu,\{\nu\})$. But then $\mu\alpha(\tau\eta)(\M,\M\vee d(\tau))=\nu\tau\eta$ which implies $\MCE(\mu,\nu\tau)\neq \emptyset$, a contradiction to the fact that $\tau$ is an entrance to $(\mu,\nu)$. Thus, \[(\tau\eta)(0,\M)\in v\Lambda^\M\setminus \image(\varphi).\] Now, using the previous argument, we have \[v(-\N)=\displaystyle{\sum_{\beta\in \image(\varphi)}}s(\beta)<s((\tau\eta)(0,\M))+\displaystyle{\sum_{\beta\in \image(\varphi)}}s(\beta)\le \displaystyle{\sum_{\beta\in v\Lambda^\M}} s(\beta)=v(-\M),\] and consequently, $^{d(\mu)}v<~^{d(\nu)}v$. If $(\mu,\nu)$ has no entrance, then clearly $(\nu,\mu)$ is a generalized cycle. An exact same argument as used in the proof of the first part now gives $^{d(\nu)}v\le~^{d(\mu)}v$. Since the relation $\le$ on $T_\Lambda$ is antisymmetric, it follows that $^{d(\mu)}v=~^{d(\nu)}v$. Therefore, if $^{d(\mu)}v<~^{d(\nu)}v$, the generalized cycle $(\mu,\nu)$ must have an entrance. 
\end{proof}
We now state and prove the following theorem, the first two parts of which are immediate consequences of the preceding lemma, whereas the final part gives a talented monoid criterion to have a generalized cycle without entrance. 
\begin{thm}\label{th existence of generalized cycle via TM}
Suppose $\Lambda$ is a row-finite $k$-graph without sources. Then the following hold.

$(i)$ If $\Lambda$ contains a generalized cycle, then there exist $0\neq a\in T_\Lambda$ and $\N,\M\in \mathbb{N}^k$ with $\N\neq \M$ such that $^\N a\le~^\M a$. 

$(ii)$ If $\mathbb{Z}^k$ acts freely on $T_\Lambda$, then every generalized cycle in $\Lambda$ has an entrance.

$(iii)$ If there is a periodic atom in $T_\Lambda$, then there is a generalized cycle in $\Lambda$ which has no entrance. Consequently, the $C^*$-algebra $C^*(\Lambda)$ is not AF. 
\end{thm}
\begin{proof}
$(i)$ Suppose $(\mu,\nu)$ is a generalized cycle in $\Lambda$. Then $d(\mu)\neq d(\nu)$. Letting $a=s(\mu)$, $\N=d(\mu)$ and $\M=d(\nu)$, Lemma \ref{lem generalized cycle} yields $^\N a\le~^\M a$. 

$(ii)$ Let $(\mu,\nu)$ be any generalized cycle in $\Lambda$. Then by the first part of Lemma \ref{lem generalized cycle}, we have $^{d(\mu)}s(\mu)\le ~^{d(\nu)}s(\mu)$. Since $d(\mu)\neq d(\nu)$, the equality cannot hold by the condition otherwise $s(\mu)(0)$ would be a periodic element in $T_\Lambda$. Therefore, $^{d(\mu)}s(\mu)<~^{d(\nu)}s(\mu)$. The second part of Lemma \ref{lem generalized cycle} now guarantees an entrance to the generalized cycle $(\mu,\nu)$.

$(iii)$ Suppose $a$ is a periodic atom in $T_\Lambda$. Then $a$ must be of the form $a=v(\mathbf{t})$ for some $v\in \Lambda^0$ and $\mathbf{t}\in \mathbb{Z}^k$ such that $|v\Lambda^\mathbf{q}|=1$ for all $\mathbf{q}\in \mathbb{N}^k$. Now since $a$ is periodic, there exists $\bP\in \mathbb{Z}^k\setminus \{0\}$ such that $^\bP a=a$. This amounts to say that there exist distinct $\N,\M\in \mathbb{N}^k$ with $^\N a=~^\M a$. Then $v(\N)=~^{-\mathbf{t}}(^\N a)=~^{-\mathbf{t}}(^\M a)=v(\M)$. Transferring this equality to $M_{\overline{\Lambda}}$ and applying the confluence lemma, we get some $\gamma\in \mathbb{F}_{\overline{\Lambda}}\setminus \{0\}$ such that $(v,\N)\longrightarrow \gamma$ and $(v,\M)\longrightarrow \gamma$. Since $|v\Lambda^\mathbf{q}|=1$ for all $\mathbf{q}\in \mathbb{N}^k$, this forces that $\gamma=\overline{s}(\lambda)=\overline{s}(\mu)$ for some $\alpha\in (v,\N)\overline{\Lambda}$ and $\beta\in (v,\M)\overline{\Lambda}$. Then $\alpha=(\mu,\N)$ and $\beta=(\nu,\M)$ for some $\mu,\nu\in \Lambda$ with $r(\mu)=r(\nu)=v$, $s(\mu)=s(\nu)$ and $\N+d(\mu)=\M+d(\nu)$. Since $\N\neq \M$, $d(\mu)\neq d(\nu)$ which implies $\mu\neq \nu$. Now, we claim that $(\mu,\nu)$ is our desired generalized cycle. It is enough to show that $(\mu,\nu)$ is a generalized cycle. That it has no entrance will follow from Lemma \ref{lem generalized cycle} by virtue of the fact that \[^{d(\mu)}s(\mu)=~^{-\N}s(\mu)(\N+d(\mu))=v(0)=~^{-\M}s(\mu)(\M+d(\nu))=~^{d(\nu)}s(\mu).\] Pick any $\tau\in s(\mu)\Lambda$. We need to show that $\MCE(\mu\tau,\nu)\neq \emptyset$. Seeking a contradiction, assume that $\MCE(\mu\tau,\nu)=\emptyset$. Let $\mathbf{q}=(d(\mu)+d(\tau))\vee d(\nu)$ and choose any $\eta\in s(\tau)\Lambda^{\mathbf{q}-(d(\mu)+d(\tau))}$ and $\delta\in s(\mu)\Lambda^{\mathbf{q}-d(\nu)}$. Then in $T_\Lambda$, \[a=v(\mathbf{t})=\displaystyle{\sum_{\lambda\in v\Lambda^{\mathbf{q}}}} s(\lambda)(\mathbf{t}+\mathbf{q})\ge s(\eta)(\mathbf{t}+\mathbf{q})+s(\delta)(\mathbf{t}+\mathbf{q})\] since $\mu\tau\eta\neq \nu\delta$. This clearly contradicts the fact that $a$ is an atom. Thus, $\MCE(\mu\tau,\nu)\neq \emptyset$ for all $\tau\in s(\mu)\Lambda$, and so $(\mu,\nu)$ is generalized cycle. The final part now follows from \cite[Theorem 3.4]{Evans}.     
\end{proof}

Evans and Sims in \cite[Corollary 3.8]{Evans} proved that $C^*(\Lambda)$ is not AF provided there exists a generalized cycle in $\Lambda$ with an entrance. We now obtain a purely algebraic analogue of their result incorporating the talented monoid. 

\begin{prop}\label{pro a necessary condition for KPA to be ultramatricial}
Let $\Lambda$ be a row-finite $k$-graph without sources. Suppose $(\mu,\nu)$ is a generalized cycle in $\Lambda$. If $\mathbb{Z}^k$ acts freely on $T_\Lambda$, then the Kumjian--Pask algebra $\KP_\mathsf{F}(\Lambda)$ is not ultramatricial. 
\end{prop}
\begin{proof}
Note that the idempotents $s_\mu s_{\mu^*}$ and $s_\nu s_{\nu^*}$ of $\KP_\mathsf{F}(\Lambda)$ are Murray-von Neumann equivalent since $s_\mu s_{\mu^*}=s_\mu s_{\nu^*}s_\nu s_{\mu^*}$ and $s_\nu s_{\nu^*}=s_\nu s_{\mu^*} s_\mu s_{\nu^*}$. We claim that $s_\mu s_{\mu^*}< s_\nu s_{\nu^*}$. For this first we observe that $s_\mu s_{\mu^*}\leqslant s_\nu s_{\nu^*}$. The idea is same as that of \cite[Lemma 3.7]{Evans}, but we reproduce it here in the algebraic set-up. Let $\N,\M$ be as defined in the proof of Lemma \ref{lem generalized cycle}. Since $\Ext(\mu,\{\nu\})$ is exhaustive, following the proof of Lemma \ref{lem generalized cycle}, we have $s(\mu)\Lambda^\N=\Ext(\mu,\{\nu\})$. Now using the relation $(KP4)$ of the Kumjian--Pask algebra, we have \[p_{s(\mu)}=\displaystyle{\sum_{\alpha\in s(\mu)\Lambda^\N}} s_\alpha s_{\alpha^*}=\displaystyle{\sum_{\alpha\in \Ext(\mu,\{\nu\})}} s_\alpha s_{\alpha^*}.\] Then \[s_\mu s_{\mu^*}=s_\mu p_{s(\mu)} s_{\mu^*}=\displaystyle{\sum_{\alpha\in \Ext(\mu,\{\nu\})}} s_{\mu\alpha}s_{(\mu\alpha)^*}=\displaystyle{\sum_{\beta\in \image(\varphi)}} s_{\nu\beta}s_{(\nu\beta)^*},\] where $\varphi$ is the map defined in Lemma \ref{lem generalized cycle}. Again, we have \[s_\nu s_{\nu^*}=s_\nu p_{s(\nu)} s_{\nu^*}=\displaystyle{\sum_{\beta\in s(\nu)\Lambda^\M}} s_{\nu\beta}s_{(\nu\beta)^*}.\] Since $\image(\varphi)\subseteq s(\nu)\Lambda^\M$ and $s_{(\nu\beta)^*}s_{\nu\beta'}=0$ for all $\beta,\beta'\in s(\nu)\Lambda^\M$ with $\beta\neq \beta'$ (by $(KP3)$), \[s_\mu s_{\mu^*}s_\nu s_{\nu^*}=s_{\nu}s_{\nu^*}s_\mu s_{\mu^*}=s_\mu s_{\mu^*}.\] Hence, $s_\mu s_{\mu^*}\leqslant s_\mu s_{\nu^*}$. By Proposition \ref{th existence of generalized cycle via TM} $(ii)$, the condition in the statement implies that $(\mu,\nu)$ has an entrance. Let $\tau$ be an entrance. Using $(KP4)$ again, we can write \[s_\nu s_{\nu^*}=\displaystyle{\sum_{\eta\in s(\nu)\Lambda^{d(\tau)}}} s_{\nu\eta}s_{(\nu\eta)^*}.\] Since the summand are orthogonal and $s_{\nu\tau}s_{(\nu\tau)^*}$ is a summand, it follows that $s_{\nu\tau}s_{(\nu\tau)^*}\leqslant s_\nu s_{\nu^*}$. Now, since $\tau$ is an entrance to $(\mu,\nu)$, $\Lambda^{\min}(\mu,\nu\tau)=\emptyset$. This implies \[s_\mu s_{\mu^*}s_{\nu\tau} s_{(\nu\tau)^*}=s_\mu\left(\displaystyle{\sum_{(\gamma,\delta)\in \Lambda^{\min}(\mu,\nu\tau)}}s_\gamma s_{\delta^*}\right)s_{(\nu\tau)^*}=0.\] Similarly, $s_{\nu\tau}s_{(\nu\tau)^*} s_\mu s_{\mu^*}=0$. So $(s_\mu s_{\mu^*}+s_{\nu\tau}s_{(\nu\tau)^*})$ is an idempotent and clearly, $(s_\mu s_{\mu^*}+s_{\nu\tau}s_{(\nu\tau)^*})\leqslant s_\nu s_{\nu^*}$. Again $s_\mu s_{\mu^*}< (s_\mu s_{\mu^*}+s_{\nu\tau}s_{(\nu\tau)^*})$ since $s_{\nu\tau} s_{(\nu\tau)^*}\neq 0$. Combining the last two inequalities, the claim is established. Therefore, $s_\nu s_{\nu^*}\sim s_\mu s_{\mu^*}< s_\nu s_{\nu^*}$ and hence $s_\nu s_{\nu^*}$ is an infinite idempotent in $\KP_\mathsf{F}(\Lambda)$. Since ultramatricial algebras cannot contain infinite idempotents, the result follows.   
\end{proof}


Let us again refer to Example \ref{ex KP is purely infinite simple via TM}. Note that $\Lambda$ contains cycles and hence generalized cycles. For instance, consider the generalized cycle $(f_1,v)$. Since $\mathbb{Z}^2$ acts freely on $T_\Lambda$, this has an entrance, namely $f_2$ or $f_3$. The argument used in the proof of Theorem \ref{pro a necessary condition for KPA to be ultramatricial} now shows that the idempotent $p_v\in \KP_\mathsf{F}(\Lambda)$ is an infinite idempotent, and consequently, $\KP_\mathsf{F}(\Lambda)$ is not ultramatricial. It is in fact purely infinite as we see in Example \ref{ex KP is purely infinite simple via TM}. 

\section{Locally finite, and crossed product Kumjian--Pask algebras}\label{sec locally finite and crossed product}
In this section, we show that the talented monoid can effectively characterize when the $\mathbb{Z}^k$-graded Kumjian--Pask algebra of a $k$-graph is locally finite, and also when it is a crossed product. In addition to this, we also introduce a higher-dimensional analogue of the geometric condition of a directed graph producing crossed product Leavitt path algebras, and prove that this, along with the Condition (HNE) defined in Subsection \ref{ssec loops}, is sufficient for the Kumjian--Pask algebra to be a crossed product. However, it is not necessary unlike the one-dimensional case, as we show in Example \ref{ex HEDL is not necessary} making the structural divergence between directed graphs and general $k$-graphs more prominent. 

\subsection{Talented monoid criterion for locally finite Kumjian--Pask algebras}\label{ssec locally finite} The authors in \cite{AAS} have shown that the Leavitt path algebra of a directed graph $E$ is locally finite if and only if $E$ is finite and has Condition (NE) (see \cite[Theorem 1.8]{AAS}). The next proposition says that these conditions are also necessary for the Kumjian--Pask algebra of a $k$-graph to be locally finite as a $\mathbb{Z}^k$-graded algebra. 
\begin{prop}\label{pro Necessary condition for locally finite KPA}
Let $\Lambda$ be a row-finite $k$-graph without sources. If $\KP_\mathsf{F}(\Lambda)$ is locally finite, then $|\Lambda^0|< \infty$ and $\Lambda$ has Condition \textup{(HNE)}.    
\end{prop}
\begin{proof}
Suppose $\KP_\mathsf{F}(\Lambda)$ is locally finite. Clearly $\Lambda^0$ is a linearly independent subset of $\KP_\mathsf{F}(\Lambda)_0$. Hence, $|\Lambda^0|\le \dim_\mathsf{F}(\KP_\mathsf{F}(\Lambda)_0)<\infty$. Now, we show that $\Lambda$ must have Condition (HNE). If possible, suppose there is a loop $\mu$ in $\Lambda$ which has an entrance $\lambda$. Let $d(\mu)=\N$. Consider the set $X:=\{s_\mu^{i+1}s_{\mu^*}^i~|~i\in \mathbb{N}\}$. Then $X\subseteq \KP_\mathsf{F}(\Lambda)_\N$. We claim that $S$ is linearly independent. Let $S=\{s_\mu^{i_1+1}s_{\mu^*}^{i_1},s_\mu^{i_2+1}s_{\mu^*}^{i_2},\ldots,s_\mu^{i_t+1}s_{\mu^*}^{i_t}\}$ be any finite subset of $X$ with $i_1<i_2<\ldots <i_t$. Suppose $\displaystyle{\sum_{j=1}^{t}}c_js_\mu^{i_j+1}s_{\mu^*}^{i_j}=0$ for scalars $c_j\in \mathsf{F}$. Then multiplying on the left by $s_{\mu^*}^{i_1+1}$ and on the right by $s_\mu^{i_1}$, we get \[c_1p_{s(\mu)}+\displaystyle{\sum_{j=2}^{t}}c_js_\mu^{i_j-i_1}s_{\mu^*}^{i_j-i_1}=0,\] which when multiplying on the right by $s_\lambda$ gives \[c_1s_\lambda+\displaystyle{\sum_{j=2}^{t}}c_js_\mu^{i_j-i_1}s_{\mu^*}^{i_j-i_1}s_\lambda=0.\] Now, since $\lambda$ is an entrance to $\mu$, $d(\lambda)\le d(\mu)$ and $\mu(0,d(\lambda))\neq \lambda$. This implies for each $2\le j\le t$, \[s_\mu^{i_j-i_1}s_{\mu^*}^{i_j-i_1}s_\lambda=s_\mu^{i_j-i_1}s_{\mu^*}^{i_j-i_1-1}(s_{\mu^*}s_\lambda)=s_\mu^{i_j-i_1}s_{\mu^*}^{i_j-i_1-1}(s_{\mu(d(\lambda),d(\mu)-d(\lambda))^*}s_{\mu(0,d(\lambda))^*}s_\lambda)=0\] by Condition $(KP3)$ of the Kumjian--Pask algebra. Therefore, we have $c_1 p_{s(\lambda)}=s_{\lambda}^*(c_1s_\lambda)=0$ and so $c_1=0$ by \cite[Theorem 3.4]{Pino}. Following the exact similar argument in each case, it can be shown that $c_j=0$ for all $j=2,3,\ldots,t$. Thus, our claim is established. Now, this implies $\dim_\mathsf{F}(\KP_\mathsf{F}(\Lambda)_\N)\ge |X|=\infty$, contradicting that $\KP_\mathsf{F}(\Lambda)$ is locally finite. 
\end{proof}

We now proceed to show that the conditions on the $k$-graph mentioned in Proposition \ref{pro Necessary condition for locally finite KPA} are also sufficient for Kumjian--Pask algebras to be locally finite. In order to accomplish this, we need the following lemma which can be seen as a higher-dimensional extension of \cite[Lemma 1.3]{AAS}. 

Let us first fix some notations. Suppose $\Lambda$ is a row-finite $k$-graph with no sources and $\N\in \mathbb{Z}^k$. For any $\M\in \mathbb{N}^k$ with $\M\ge \N$, define \[\mS_\M^\N:=\{s_\alpha s_{\beta^*}\in \KP_\mathsf{F}(\Lambda)~|~d(\alpha)=\M~\text{and}~d(\beta)=\M-\N\}\cup \{0\}.\] If $|\Lambda^0|< \infty$, then each $\mS_\M^\N$ is a finite set. Note that in view of \cite[(3.3)]{Pino}, we have 
\begin{equation}\label{eq the span}
\KP_\mathsf{F}(\Lambda)_\N=\lspan_\mathsf{F} \{s_\alpha s_{\beta^*}~|~\alpha,\beta\in \Lambda~\text{and}~d(\beta)=d(\alpha)-\N\}=\lspan_\mathsf{F}\left(\displaystyle{\bigcup_{\M\ge \N}}\mS_\M^\N\right).
\end{equation}
For two $k$-tuples of integers $\bP=(p_1,p_2,\ldots,p_k)$ and $\Q=(q_1,q_2,\ldots,q_k)$, we define \[\sur(\bP,\Q):=\displaystyle{\sum_{i=1}^{k}}\max(p_i-q_i,0).\] It is easy to observe that $\bP\le \Q$ if and only if $\sur(\bP,\Q)=0$. 
\begin{lem}\label{lem the absorbing lemma}
\textup{($cf.$ \cite[Lemma 1.3]{AAS})} Let $\Lambda$ be a row-finite $k$-graph without sources and $\N\in \mathbb{Z}^k$. If there exists $\bT=(t_1,t_2,\ldots,t_k)\in \mathbb{N}^k$ such that $\bT\ge \N$ and $\mS_\Q^\N\subseteq \displaystyle{\bigcup_{\N\le \M\le \bT}}\mS_\M^\N$ for all $\Q\ge \N$ with $\sur(\Q,\bT)=1$, then $\displaystyle{\bigcup_{\M\ge \N}}\mS_\M^\N=\displaystyle{\bigcup_{\N\le \M\le \bT}} \mS_\M^\N$. Consequently, if $|\Lambda^0|< \infty$, then $\KP_\mathsf{F}(\Lambda)_\N$ is finite-dimensional. 
\end{lem}
\begin{proof}
We need to show that $\mS_\Q^\N\subseteq \displaystyle{\bigcup_{\N\le \M\le \bT}} \mS_\M^\N$ for all $\Q\ge \N$. For this, we use induction on the counter $\sur(\Q,\bT)$. If $\sur(\Q,\bT)=0$, then $\Q\le \bT$ and there is nothing to prove. If $\sur(\Q,\bT)=1$, then the result follows from the condition of the Lemma. Hence, we are done for the two base cases. Now, assume that the result holds for any $\Q\ge \N$ with $\sur(\Q,\bT)=l$ for some $l\ge 1$. Choose any $\Q=(q_1,q_2,\ldots,q_k)\ge \N$ with $\sur(\Q,\bT)=l+1$. Then there exists at least one $j\in \{1,2,\ldots,k\}$ such that $q_j-t_j\ge 1$. Let $s_\alpha s_{\beta^*}\in \mS_\Q^\N$. Then $d(\alpha(\mathbf{e}_j,d(\alpha)))=\Q-\mathbf{e}_j$ and $d(\beta(\mathbf{e}_j,d(\beta)))=\Q-\mathbf{e}_j-\N$. Since $\Q,\bT\ge \N$ and $q_j-t_j\ge 1$, $\Q-\mathbf{e}_j\ge \N$. Also \[\sur(\Q-\mathbf{e}_j,\bT)=\displaystyle{\sum_{i=1}^{k}} \max(q_i-\delta_{ji}-t_i,0)=\left(\displaystyle{\sum_{i=1}^{k}}\max(q_i-t_i,0)\right)-1=\sur(\Q,\bT)-1=l.\] Thus, by the induction hypothesis, $s_{\alpha(\mathbf{e}_j,d(\alpha))}s_{\beta(\mathbf{e}_j,d(\beta))^*}\in \mS_{\Q-\mathbf{e}_j}^\N\subseteq \displaystyle{\bigcup_{\N\le \M\le \bT}} \mS_\M^\N$. Then 
\begin{align*}
s_\alpha s_{\beta^*}&=s_{\alpha(0,\mathbf{e}_j)}(s_{\alpha(\mathbf{e}_j,d(\alpha))}s_{\beta(\mathbf{e}_j,d(\beta))^*})s_{\beta(0,\mathbf{e}_j)^*}\\
&\subseteq s_{\alpha(0,\mathbf{e}_j)}\left(\displaystyle{\bigcup_{\N\le \M\le \bT}} \mS_\M^\N\right)s_{\beta(0,\mathbf{e}_j)^*}\\
&\subseteq \displaystyle{\bigcup_{\N+\mathbf{e}_j\le \M\le \bT+\mathbf{e}_j}} \mS_\M^\N=\left(\displaystyle{\bigcup_{\substack{\N+\mathbf{e}_j\le \M\le \bT+\mathbf{e}_j\\\sur(\M,\bT)=0}}} \mS_\M^\N\right)\cup \left(\displaystyle{\bigcup_{\substack{\N+\mathbf{e}_j\le \M\le \bT+\mathbf{e}_j\\\sur(\M,\bT)=1}}} \mS_\M^\N\right) \subseteq \displaystyle{\bigcup_{\N\le \M\le \bT}} \mS_\M^\N,
\end{align*} 
where the last inclusion follows from the two base cases. Therefore, by induction $\mS_\Q^\N\subseteq \displaystyle{\bigcup_{\N\le \M\le \bT}} \mS_\M^\N$ for all $\Q\ge \N$. The last statement is now evident from (\ref{eq the span}) by noting that $\displaystyle{\bigcup_{\N\le \M\le \bT}} \mS_\M^\N$ is a finite union of finite sets.  
\end{proof}

Now we are ready to prove the following theorem which extends \cite[Theorem 1.8]{AAS} to higher dimensions.
\begin{thm}\label{th N&S condition for locally finite KPA}
Let $\Lambda$ be a row-finite $k$-graph without sources. Then   $\KP_\mathsf{F}(\Lambda)$ is locally finite as a $\mathbb{Z}^k$-graded algebra if and only if $|\Lambda^0|< \infty$ and $\Lambda$ has Condition \textup{(HNE)}.    
\end{thm}
\begin{proof}
The necessity is already established in Proposition \ref{pro Necessary condition for locally finite KPA}. So assume that $|\Lambda^0|< \infty$ and $\Lambda$ has Condition \textup{(HNE)}. Let $\N=(n_1,n_2,\ldots,n_k)\in \mathbb{Z}^k$. We consider the $k$-tuple $\bT=(t_1,t_2,\ldots,t_k)$, where $t_i=2\times \max(n_i,|\Lambda^0|)$ for each $i=1,2,\ldots,k$. Clearly, $\bT\ge\N$. If we can show that $\mS_\Q^\N\subseteq \displaystyle{\bigcup_{\N\le \M\le \bT}}\mS_\M^\N$ for all $\Q\ge \N$ with $\sur(\Q,\bT)=1$, then we are done by Lemma \ref{lem the absorbing lemma}. So let us choose such a $\Q$. Then there exists a $j\in \{1,2,\ldots,k\}$ such that $q_j-t_j=1$ and $q_i\le t_i$ for all $i\neq j$. Let $s_\alpha s_{\beta^*}\in \mS_\Q^\N\setminus \{0\}$, $\alpha':=\alpha(d(\alpha)-q_j\mathbf{e}_j,d(\alpha))$ and $\beta':=\beta(d(\beta)-(q_j-n_j)\mathbf{e}_j,d(\beta))$. Since $\Lambda$ has no sources, $q_j=t_j+1>t_j\ge |\Lambda^0|$ and \[q_j-n_j=t_j+1-n_j=\max(n_j,2|\Lambda^0|-n_j)+1>|\Lambda^0|,\] there must exist cycles $\mu$ and $\nu$ entirely made of edges of degree $\mathbf{e}_j$ such that \[\alpha'((q_j-1)\mathbf{e}_j,q_j\mathbf{e}_j)=\mu(0,\mathbf{e}_j),\,  \text{ and } \, \beta'((q_j-n_j-1)\mathbf{e}_j,(q_j-n_j)\mathbf{e}_j)=\nu(0,\mathbf{e}_j).\] Without loss of generality, assume that $d(\mu)\le d(\nu)$. We claim that $\mu(0,\mathbf{e}_j)=\nu(0,\mathbf{e}_j)$. Since $s_\alpha s_{\beta^*}\neq 0$, $s(\alpha)=s(\beta)$, whence $\mu(\mathbf{e}_j)=s(\mu(0,\mathbf{e}_j))=s(\nu(0,\mathbf{e}_j))=\nu(\mathbf{e}_j)$. Since no loop has an entrance, we must have $\mu(\mathbf{e}_j,d(\mu))=\nu(\mathbf{e}_j,d(\mu))$, otherwise $\nu(\mathbf{e}_j,d(\mu))$ would be an entrance to the loop $\sigma_{\mathbf{e}_j}(\mu)$. Now, if $d(\mu)=d(\nu)$, then \[\mu(0)=\mu(d(\mu))=s(\mu(\mathbf{e}_j,d(\mu)))=s(\nu(\mathbf{e}_j,d(\nu)))=\nu(d(\nu))=\nu(0),\] and consequently, $\mu(0,\mathbf{e}_j)=\nu(0,\mathbf{e}_j)$ as otherwise $\nu(0,\mathbf{e}_j)$ would be an entrance to $\mu$. On the other hand, if $d(\mu)< d(\nu)$, then $\mu(0,\mathbf{e}_j)=\nu(d(\mu),d(\mu)+\mathbf{e}_j)$ as otherwise $\nu(d(\mu),d(\mu)+\mathbf{e}_j)$ would be an entrance to $\mu$. But then $\nu(d(\mu)+\mathbf{e}_j)=\mu(\mathbf{e}_j)=\nu(\mathbf{e}_j)$ which is not possible since $\nu$ is a cycle (see Definition \ref{def cycle in a k-graph}). Therefore, the claim is established. Now, since $\mu$ has no entrance, $r(\mu)\Lambda^{\mathbf{e}_j}=\{\mu(0,\mathbf{e}_j)\}$. Applying Condition $(KP4)$ of the Kumjian--Pask algebra, we have \[s_\alpha s_{\beta^*}=s_{\alpha(0,d(\alpha)-\mathbf{e}_j)}(s_{\alpha'((q_j-1)\mathbf{e}_j,q_j\mathbf{e}_j)}s_{\beta'((q_j-n_j-1)\mathbf{e}_j,(q_j-n_j)\mathbf{e}_j)^*})s_{\beta(0,d(\beta)-\mathbf{e}_j)^*}=s_{\alpha(0,d(\alpha)-\mathbf{e}_j)}s_{\beta(0,d(\beta)-\mathbf{e}_j)^*}.\] Since $q_i\ge n_i$ for all $i\neq j$ and $q_j-1=t_j\ge n_j$, $d(\alpha)-\mathbf{e}_j\ge \N$. Again, $d(\alpha)-\mathbf{e}_j=\Q-\mathbf{e}_j\le \bT$. From these, it follows that \[s_\alpha s_{\beta^*}=s_{\alpha(0,d(\alpha)-\mathbf{e}_j)}s_{\beta(0,d(\beta)-\mathbf{e}_j)^*}\in \displaystyle{\bigcup_{\N\le \M\le \bT}} \mS_\M^\N\] and hence the proof. 
\end{proof}

We now merge our findings and present the following theorem showcasing another instance which connects the geometry of a $k$-graph, the algebraic structure of the Kumjian--Pask algebra and the $\mathbb{Z}^k$-monoid structure of the talented monoid. Also it helps us to explicitly describe locally finite Kumjian--Pask algebras of $k$-graphs which arise as the pullback of $1$-graphs. 
\begin{thm}\label{th LOcally finite KPA of pullback}
Let $\Lambda$ be a row-finite $k$-graph without sources such that $|\Lambda^0|<\infty$. Then the following statements are equivalent.

$(i)$ $\KP_\mathsf{F}(\Lambda)$ is locally finite.

$(ii)$ $\Lambda$ has Condition \textup{(HNE)}.

$(iii)$ For any $0\neq a\in T_\Lambda$ and $i=1,2,\ldots,k$, there exists $n_i\in \mathbb{N}\setminus \{0\}$ such that $^{n_i\mathbf{e}_i}a=a$.

Moreover, if $\Lambda=f^*(E^*)$ for a directed graph $E$ and a surjective homomorphism $f:\mathbb{N}^k\longrightarrow \mathbb{N}$, and any one of the above holds, then there exist positive integers $t$ and $m_i$, $i=1,2,\ldots,t$ such that \[\KP_\mathsf{F}(\Lambda)\cong \displaystyle{\bigoplus_{i=1}^{t}}~\mathbb{M}_{m_i}(\mathsf{F}[\mathbb{Z}^k]).\]
\end{thm}
\begin{proof}
The equivalence of the statements $(i)$--$(iii)$ follows from Corollary \ref{cor characterizing HL and HNE via TM} and Theorem \ref{th N&S condition for locally finite KPA}. Now assume that $\Lambda$ is a pullback as described in the final statement and $(ii)$ holds. Then $E$ has Condition (NE) by Corollary \ref{cor Condition HL and HNE under pullback}$(i)$. Now, since $E$ is finite and has no sinks (this is because $|\Lambda^0|<\infty$ and $\Lambda$ is row-finite with no sources), from \cite[Theorems 1.8 \& 3.8]{AAS}, it follows that 
\begin{equation}\label{eq the iso}
L_\mathsf{F}(E)\cong \displaystyle{\bigoplus_{i=1}^{t}}~\mathbb{M}_{m_i}(\mathsf{F}[\mathbb{Z}]),    
\end{equation}
where $t$ is the number of cycles in $E$ and for each $i=1,2,\ldots,t$, $m_i$ is the number of paths ending at a fixed vertex on the $i$-th cycle excluding the cycle itself. Since the homomorphism $f$ is surjective, we can apply \cite[Proposition 2.10]{KumjianPask} to obtain $\mathcal{G}_{\Lambda}\cong \mathcal{G}_{E}\times \mathbb{Z}^{k-1}$, where $\mathcal{G}_\Lambda$ is the infinite path groupoid of $\Lambda$ and $\mathcal{G}_E$ is the graph groupoid of $E$. Now, applying \cite[Theorem 4.3]{Rigby} together with the isomorphism (\ref{eq the iso}), we have \[\KP_\mathsf{F}(\Lambda)\cong A_\mathsf{F}(\mathcal{G}_\Lambda)\cong A_\mathsf{F}(\mathcal{G}_E)\otimes_\mathsf{F} A_\mathsf{F}(\mathbb{Z}^{k-1})\cong L_\mathsf{F}(E)\otimes_\mathsf{F} \mathsf{F}[\mathbb{Z}^{k-1}]\cong \displaystyle{\bigoplus_{i=1}^{t}}\left(\mathbb{M}_{m_i}(\mathsf{F}[\mathbb{Z}])\otimes_\mathsf{F} \mathsf{F}[\mathbb{Z}^{k-1}]\right).\] Finally, since the tensor product is associative upto isomorphism, \[\mathbb{M}_{m_i}(\mathsf{F}[\mathbb{Z}])\otimes_\mathsf{F} \mathsf{F}[\mathbb{Z}^{k-1}]\cong (\mathbb{M}_{m_i}(\mathsf{F})\otimes_\mathsf{F} \mathsf{F}[\mathbb{Z}])\otimes_\mathsf{F} \mathsf{F}[\mathbb{Z}^{k-1}]\cong \mathbb{M}_{m_i}(\mathsf{F})\otimes_\mathsf{F} \mathsf{F}[\mathbb{Z}^k]\cong \mathbb{M}_{m_i}(\mathsf{F}[\mathbb{Z}^k])\] and the result follows. 
\end{proof}
As an immediate consequence of the first part of the above theorem, we obtain the following interesting result which says that the graded $K$-theory is capable to distinguish the class of locally finite Kumjian--Pask algebras.
\begin{cor}\label{cor distinguishing locally finite KPAs}
Let $\Lambda$ and $\Omega$ be row-finite $k$-graphs without sources such that $|\Lambda^0|, |\Omega^0|<\infty$. Suppose any one of the following equivalent conditions holds.

$(i)$ There is a $\mathbb{Z}^k$-monoid isomorphism between $T_\Lambda$ and $T_\Omega$.

$(ii)$ There is an order-preserving $\mathbb{Z}[\mathbb{Z}^k]$-module isomorphism between $K_0^{\gr}(\KP_\mathsf{F}(\Lambda))$ and $K_0^{\gr}(\KP_\mathsf{F}(\Omega))$.

Then $\KP_\mathsf{F}(\Lambda)$ is locally finite if and only if $\KP_\mathsf{F}(\Omega)$ is locally finite.
\end{cor}

\subsection{Crossed product Kumjian--Pask algebras}\label{ssec crosspdt} In \cite{Vas}, the first author and Vas characterized finite directed graphs for which the corresponding $\mathbb{Z}$-graded Leavitt path algebras are crossed products. Our next aim is to obtain a similar kind of geometric criteria for crossed product Kumjian--Pask algebras of $k$-graphs. It is helpful to record the following theorem which says that one can distinguish the class of crossed product Kumjian--Pask algebras via the graded $K$-theory or equivalently, the talented monoid.

\begin{prop}\label{pro TM distinguishing crosspdt KPAs}
Suppose $\Lambda$ is a row-finite $k$-graph with no sources and $|\Lambda^0|<\infty$. Then $\KP_\mathsf{F}(\Lambda)$ is a crossed product (as a $\mathbb{Z}^k$-graded ring) if and only if $^\N \epsilon_\Lambda=\epsilon_\Lambda$ for all $\N\in \mathbb{Z}^k$. Consequently, if $\Lambda$ and $\Omega$ are row-finite $k$-graphs with no sources such that $|\Lambda^0|,|\Omega^0|<\infty$ and any one of the following equivalent conditions hold:

$(i)$ there is a $\mathbb{Z}^k$-monoid isomorphism $T_\Lambda\longrightarrow T_\Omega$ mapping $\epsilon_\Lambda$ to $\epsilon_\Omega$,

$(ii)$ there is an order-preserving $\mathbb{Z}[\mathbb{Z}^k]$-module isomorphism $K_0^{\gr}(\KP_\mathsf{F}(\Lambda))\longrightarrow K_0^{\gr}(\KP_\mathsf{F}(\Omega))$ mapping $[\KP_\mathsf{F}(\Lambda)]$ to $[\KP_\mathsf{F}(\Omega)]$;

then $\KP_\mathsf{F}(\Lambda)$ is a crossed product if and only if $\KP_\mathsf{F}(\Omega)$ is a crossed product. 
\end{prop}
\begin{proof}
In view of \cite[Theorem 2.10.2]{NObook}, $\KP_\mathsf{F}(\Lambda)$ is a crossed product if and only if $\KP_\mathsf{F}(\Lambda)\cong_{\gr}\KP_\mathsf{F}(\Lambda)(\N)$ for all $\N\in \mathbb{Z}^k$, where $\KP_\mathsf{F}(\Lambda)$ is viewed as the regular graded left module over itself. This is equivalent to saying that in $\mathcal{V}^{\gr}(\KP_\mathsf{F}(\Lambda))$, $^\N[\KP_\mathsf{F}(\Lambda)]=[\KP_\mathsf{F}(\Lambda)]$ for all $\N\in \mathbb{Z}^k$. Since $[\KP_\mathsf{F}(\Lambda)]$ corresponds to $\epsilon_\Lambda$ under the $\mathbb{Z}^k$-monoid isomorphism $\mathcal{V}^{\gr}(\KP_\mathsf{F}(\Lambda))\longrightarrow T_\Lambda$, the first part follows. The final statement is now an immediate consequence of the first part.  
\end{proof}

The above proposition, together with Proposition \ref{pro TM under pullbacks}, yields the following:
\begin{cor}\label{cor pullback preserves crosspdt}
Let $\Lambda$ be a row-finite $\ell$-graph without sources, $f:\mathbb{N}^k\longrightarrow \mathbb{N}^\ell$ a surjective homomorphism. Then $\KP_\mathsf{F}(\Lambda)$ is a crossed product if and only if $\KP_\mathsf{F}(f^*(\Lambda))$ is a crossed product.   
\end{cor}
\begin{proof}
Since $\Lambda^0$ and $f^*(\Lambda)^0$ are essentially the same, the isomorphism $\rho$ of Proposition \ref{pro TM under pullbacks} maps $\epsilon_{f^*(\Lambda)}$ to $\epsilon_\Lambda$. Suppose $\KP_\mathsf{F}(\Lambda)$ is a crossed product. Let $\N\in \mathbb{Z}^k$. Then $\rho(^\N \epsilon_{f^*(\Lambda)})=~^{f(\N)}\epsilon_\Lambda=\epsilon_\Lambda=\rho(\epsilon_{f^*(\lambda)})$, and consequently, $^\N \epsilon_{f^*(\Lambda)}=\epsilon_{f^*(\Lambda)}$ since $\rho$ is an injective. It follows that $\KP_\mathsf{F}(f^*(\Lambda))$ is a crossed product by Proposition \ref{pro TM distinguishing crosspdt KPAs}. The converse is similar using the surjectivity of $f$.    
\end{proof}
We now justify the use of Proposition \ref{pro TM distinguishing crosspdt KPAs} and Corollary \ref{cor pullback preserves crosspdt} via the following example. 

\begin{example}\label{ex justifying the sufficient condition} Below are the $1$-skeletons of two $2$-graphs $\Lambda$ (on the left) and $\Omega$ (on the right).

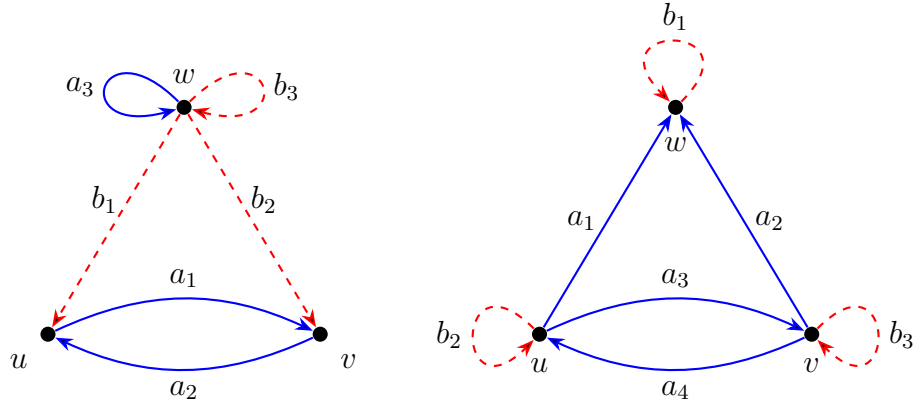
\begin{center}
\begin{tikzpicture}[
    >=Stealth,
    node distance=3.5cm,
    vertex/.style={circle, fill=black, inner sep=2pt},
    blue edge/.style={draw=blue, ->, thick},
    red edge/.style={draw=red, ->, thick}
]

    \begin{scope}[local bounding box=leftgraph]
        \node[vertex] (w1) at (0,3) {};
        \node[vertex] (u1) at (-1.8,0) {};
        \node[vertex] (v1) at (1.8,0) {};
        
        \node[above=2pt of w1] {$w$};
        \node[below left=2pt of u1] {$u$};
        \node[below right=2pt of v1] {$v$};
        
        \path[blue edge] (w1) edge[out=135, in=195, looseness=8, distance=1.5cm] node[left, black] {$a_3$} (w1);
        \path[blue edge] (u1) edge[bend left=25] node[above, black] {$a_1$} (v1);
        \path[blue edge] (v1) edge[bend left=25] node[below, black] {$a_2$} (u1);
        
        \path[red edge, dashed] (w1) edge[out=45, in=-15, looseness=8, distance=1.5cm] node[right, black] {$b_3$} (w1);
        \path[red edge, dashed] (w1) edge node[left, black, pos=0.4] {$b_1$} (u1);
        \path[red edge, dashed] (w1) edge node[right, black, pos=0.4] {$b_2$} (v1);
    \end{scope}

    \begin{scope}[shift={(6.5,0)}, local bounding box=rightgraph]
        \node[vertex] (w2) at (0,3) {};
        \node[vertex] (u2) at (-1.8,0) {};
        \node[vertex] (v2) at (1.8,0) {};
        
        \node[below=4pt of w2] {$w$};
        \node[below =2pt of u2] {$u$};
        \node[below =2pt of v2] {$v$};
        
        \path[blue edge] (u2) edge node[left, black, pos=0.5] {$a_1$} (w2);
        \path[blue edge] (v2) edge node[right, black, pos=0.5] {$a_2$} (w2);
        \path[blue edge] (u2) edge[bend left=25] node[above, black] {$a_3$} (v2);
        \path[blue edge] (v2) edge[bend left=25] node[below, black] {$a_4$} (u2);
        
        \path[red edge, dashed] (w2) edge[out=45, in=135, looseness=8, distance=1.5cm] node[above, black] {$b_1$} (w2);
        \path[red edge, dashed] (u2) edge[out=135, in=225, looseness=8, distance=1.5cm] node[left, black] {$b_2$} (u2);
        \path[red edge, dashed] (v2) edge[out=45, in=-45, looseness=8, distance=1.5cm] node[right, black] {$b_3$} (v2);
    \end{scope}
\end{tikzpicture}
\end{center}
Blue edges have degree $\mathbf{e}_1$ and dashed red edges have degree $\mathbf{e}_2$. The factorizations of bicolored paths are uniquely determined in both cases. In $T_\Lambda$, $u=v=w=w(\mathbf{e}_2)$ and $u=v(\mathbf{e}_1)$, $v=u(\mathbf{e}_1)$, $w=w(\mathbf{e}_1)$. It follows that $T_\Lambda\cong \mathbb{N}$ as $\mathbb{Z}^2$-monoids, where the $\mathbb{Z}^2$ actions are trivial. Therefore $\KP_\mathsf{F}(\Lambda)$ is a crossed product by the first part of Proposition \ref{pro TM distinguishing crosspdt KPAs}. Note that the $2$-graph $\Omega$ is the pullback of the $C_2$-comet graph 
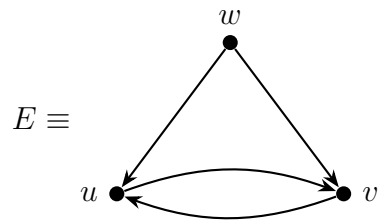
\begin{center}
\begin{tikzpicture}[
    >=Stealth, 
    vertex/.style={circle, fill=black, inner sep=2pt}, 
    edge/.style={->, thick, draw=black}
]

    \node[vertex, label=above:{$w$}] (w) at (0, 2) {};
    \node[vertex, label=left:{$u$}] (u) at (-1.5, 0) {};
    \node[vertex, label=right:{$v$}] (v) at (1.5, 0) {};

    \path[edge] (w) -- (u); 
    \path[edge] (w) -- (v); 
    
    \path[edge] (u) edge[bend left=20] (v); 
    \path[edge] (v) edge[bend left=20] (u);

\node at (-2.5,1) {$E\equiv$};

\end{tikzpicture}
\end{center}
under the homomorphism $f:\mathbb{N}^2\longrightarrow \mathbb{N}$ defined by $f((x,y))=x$. Since this graph has no exits and satisfies Condition (EDL) of \cite{Vas}, $L_\mathsf{F}(E)$ is a crossed product. Therefore, $\KP_\mathsf{F}(\Omega)$ is also a crossed product in view of Corollary \ref{cor pullback preserves crosspdt}. Note that we cannot arrive at this conclusion using the final part of Proposition \ref{pro TM distinguishing crosspdt KPAs} since the talented monoids $T_\Lambda$ and $T_\Omega$ are not isomorphic: $T_\Lambda\cong \mathbb{N}$ while $T_\Omega\cong \mathbb{N}\oplus \mathbb{N}$.  
\end{example}

Note that both the $2$-graphs $\Lambda$ and $\Omega$ have Condition (HNE). This is not a mere coincidence. The following lemma, together with Proposition \ref{pro TM distinguishing crosspdt KPAs}, will help us proving that the satiation of Condition (HNE) is a necessary condition for the Kumjian--Pask algebra to be a crossed product.

\begin{lem}\label{lem downset of periodic element}
Suppose $\Lambda$ is a row-finite $k$-graph without sources and $0\neq a\in T_\Lambda$ is such that $^\M a=a$ for some $\M\in \mathbb{N}^k\setminus \{0\}$. Then for any $b\in T_\Lambda$ with $0\neq b\le a$, there exists $t\in \mathbb{N}\setminus \{0\}$ such that $^{t\M} b=b$. 
\end{lem}
\begin{proof}
Let $b\in T_\Lambda$ be such that $b\le a$. Then $a=b+c$ for some $c\in T_\Lambda$. Moving to $M_{\overline{\Lambda}}$ and using the Confluence lemma, we have a $\eta\in \mathbb{F}_{\overline{\Lambda}}\setminus \{0\}$ such that $a\longrightarrow \eta$ and $b+c\longrightarrow \eta$. Again since $^\M a=a$, we have a $\delta\in \mathbb{F}_{\overline{\Lambda}}\setminus \{0\}$ such that $a\longrightarrow \delta$ and $^\M a\longrightarrow \delta$. Note that we are denoting elements of $T_\Lambda$ and their images in the $k$-graph monoid $M_{\overline{\Lambda}}$ by the same letters. Since $M_{\overline{\Lambda}}=\mathbb{F}_{\overline{\Lambda}}/\sim$, where $\sim$ is the congruence generated by $\longrightarrow$, it follows that $\eta=\delta$ in $M_{\overline{\Lambda}}$, and so there is a $\gamma\in \mathbb{F}_{\overline{\Lambda}}\setminus \{0\}$ such that $\eta,\delta\longrightarrow \gamma$. Then $b+c,a,~^\M a\longrightarrow \gamma$. We can write \[\gamma=\displaystyle{\sum_{j=1}^{p}} (u_j,\mathbf{q})\] for some vertices $u_j\in \Lambda^0$ and a $\Q\in \mathbb{N}^k$. Now we run the proof of Proposition \ref{pro neccesary condition for loop with entrance} with $\N=-\M$. By doing this, we obtain a fixed positive integer $t$ such that for each $j=1,2,\ldots,p$, there is a nontrivial loop $\mu_j\in u_j\Lambda^{t\M}$ which has no entrance, whence $^{t\M}u_j=u_j$ in $T_\Lambda$. Since $b+c\longrightarrow \gamma$, by \cite[Lemma 3.1]{HMPS1}, there exists $\gamma_1,\gamma_2\in \mathbb{F}_{\overline{\Lambda}}$ such that $\gamma_1\neq 0$, $b\longrightarrow \gamma_1$, $c\longrightarrow \gamma_2$ and $\gamma=\gamma_1+\gamma_2$. Then there exists $S\subseteq \{1,2,\ldots,p\}$ such that $\gamma_1=\displaystyle{\sum_{j\in S}} (u_j,\Q)$. This gives $b=\displaystyle{\sum_{j\in S}} u_{j}(\Q)$. Now, \[^{t\M}b=\displaystyle{\sum_{j\in S}}~^{t\M}u_{j}(\Q)=\displaystyle{\sum_{j\in S}} u_{j}(\Q)=b,\] which completes the proof.  
\end{proof}
\begin{cor}\label{cor HNE is necessary to be crosspdt}
Let $\Lambda$ be a row-finite $k$-graph without sources such that $|\Lambda^0|< \infty$. If $\KP_\mathsf{F}(\Lambda)$ is a crossed product, then $\Lambda$ satisfies Condition \textup{(HNE)}.
\end{cor}
\begin{proof}
Suppose $\KP_\mathsf{F}(\Lambda)$ is a crossed product. By Proposition \ref{pro TM distinguishing crosspdt KPAs}, $^\N \epsilon_\Lambda=\epsilon_\Lambda$ for all $\N\in \mathbb{Z}^k$. Let $0\neq a\in T_\Lambda$ and $i\in \{1,2,\ldots,k\}$. We claim that there exists $t\in \mathbb{N}\setminus \{0\}$ such that $^{t\mathbf{e}_i}a=a$. Since $\epsilon_\Lambda$ is an order-unit of $T_\Lambda$ and $\epsilon$ is a global fixed point under the $\mathbb{Z}^k$-action, there exists a positive integer $\ell$ such that $a\le \ell \epsilon_\Lambda$. The claim is now established by noting that $^{\mathbf{e}_i}(\ell\epsilon_\Lambda)=\ell(~^{\mathbf{e}_i}\epsilon_\Lambda)=\ell \epsilon_\Lambda$ and using Lemma \ref{lem downset of periodic element}. Therefore, $\Lambda$ has Condition (HNE) in view of Corollary \ref{cor characterizing HL and HNE via TM}.  
\end{proof}
\begin{rmk}\label{rem HNE is not sufficient for crosspdt}
Although (HNE) is a necessary condition to be a crossed product, it is not sufficient. One can find many counterexamples. For example, consider a directed graph $E$ with no sinks, with Condition (NE) and without the Condition (EDL) of \cite{Vas}. If $\Lambda$ is the pullback $k$-graph of $E$ via a surjective homomorphism $f:\mathbb{N}^k\longrightarrow \mathbb{N}$, then $\Lambda$ has Condition (HNE). But since $L_\mathsf{F}(E)$ is not a crossed product (see \cite[Theorem 3.1]{Vas}), $\KP_\mathsf{F}(\Lambda)$ is also not a crossed product by Corollary \ref{cor pullback preserves crosspdt}. For a particular example of such a $2$-graph see Example \ref{ex HEDL cannot be removed} below. 
\end{rmk}

We show that although, Condition (HNE) alone is not sufficient for the Kumjian--Pask algebra to be a crossed product, it becomes so under a specific geometric property of a $k$-graph, which we define next.

\begin{dfn}\label{def Condition HEDL}
Suppose $\Lambda$ is a row-finite $k$-graph without sources. For each $i=1,2,\ldots,k$, suppose $\mathcal{C}_i$ denotes the set of unicolor cycles of $\Lambda$ in degree $\mathbf{e}_i$. For $\mu\in \mathcal{C}_i$ and $0\le t\le |\mu|-1$, suppose $\mathcal{O}_t^\mu$ denotes the set of all paths $\lambda\in \Lambda s(\mu)$ such that $\lambda$ is unicolor in degree $\mathbf{e}_i$, $|\lambda|\equiv t ~(\text{mod}~|\mu|)$ and $\lambda(\N,\N+d(\mu))\neq \mu$ for all $0\le \N\le d(\lambda)-d(\mu)$ whenever $d(\lambda)\ge d(\mu)$. We say that the $k$-graph $\Lambda$ has Condition (HEDL) if for each $i=1,2,\ldots,k$ and any $\mu\in \mathcal{C}_i$, there exists a positive integer $\ell^\mu$ such that $|\mathcal{O}_t^\mu|=\ell^\mu$ for all $t=0,1,\ldots,|\mu|-1$.      
\end{dfn}

The following lemma will serve our purpose.
\begin{lem}\label{lem vertex representation as sum cycle elements}
Let $\Lambda$ be a row-finite $k$-graph without sources such that $|\Lambda^0|<\infty$ and $\Lambda$ has Condition \textup{(HNE)}. Let $v\in \Lambda^0$. Then for each $i=1,2,\ldots,k$, \[v(0)=\displaystyle{\sum_{\mu\in \mathcal{C}_i}}\left(\displaystyle{\sum_{t=0}^{|\mu|-1}}j_t^{\mu,v}s(\mu)(t\mathbf{e}_i)\right)\] in $T_\Lambda$, where $j_t^{\mu,v}=|v\Lambda \cap \mathcal{O}_t^\mu|$, $\mathcal{C}_i$ and $\mathcal{O}_t^\mu$ are as described in Definition \textup{\ref{def Condition HEDL}}.  
\end{lem}
\begin{proof}
Choose any $i\in \{1,2,\ldots,k\}$. Since $\Lambda$ has no sources, we can transform the vertex $(v,0)\in \overline{\Lambda}^0$ in the free monoid $\mathbb{F}_{\overline{\Lambda}}$ via the relation $\longrightarrow_i$ and also do the same for all its ancestors. After a certain number of steps, since $|\Lambda^0|<\infty$, we will obtain an element in $\mathbb{F}_{\overline{\Lambda}}$ such that all the vertices in its support is the source of some unicolor cycle in degree $\mathbf{e}_i$. Coming back to $T_\Lambda$, this amounts to say that 
\begin{equation}\label{eq vsplit}
v=\displaystyle{\sum_{\mu\in \mathcal{C}_i}}\left(\displaystyle{\sum_{\lambda\in v\Lambda\cap \mathcal{O}^\mu}}s(\mu)(|\lambda|\mathbf{e}_i)\right),
\end{equation}
where we let $\mathcal{O}^\mu$ to be the disjoint union of the sets $\mathcal{O}_t^\mu$, $t=0,1,\ldots,|\mu|-1$. We observe that for any $\mu\in \mathcal{C}_i$ and $n\in \mathbb{N}$, $s(\mu)(n\mathbf{e}_i)=s(\mu)(t\mathbf{e}_i)$, where $0\le t\le |\mu|-1$ and $n\equiv t~(\text{mod}~|\mu|)$. Indeed, writing $n=|\mu|q+t$ for some $q\in \mathbb{N}$, we have \[s(\mu)(n\mathbf{e}_i)=s(\mu)(qd(\mu)+t\mathbf{e}_i)=~^{qd(\mu)}s(\mu)(t\mathbf{e}_i)=s(\mu)(t\mathbf{e}_i),\] where the third equality follows since $\Lambda$ has Condition (HNE) and consequently, $\mu$ has no entrance (see the proof of Proposition \ref{pro sufficient condition for loop with entrance}). Now using this observation in (\ref{eq vsplit}), we have \[v=\displaystyle{\sum_{\mu\in \mathcal{C}_i}}\left(\displaystyle{\sum_{t=0}^{|\mu|-1}}\left(\displaystyle{\sum_{\lambda\in v\Lambda\cap \mathcal{O}_t^\mu}}s(\mu)(t\mathbf{e}_i)\right)\right)=\displaystyle{\sum_{\mu\in \mathcal{C}_i}}\left(\displaystyle{\sum_{t=0}^{|\mu|-1}}j_t^{\mu,v} s(\mu)(t\mathbf{e}_i)\right).\qedhere\]  
\end{proof}
Now we are ready to prove the following. 
\begin{thm}\label{th sufficient condition for crosspdt}
Let $\Lambda$ be a row-finite $k$-graph without sources such that $|\Lambda^0|<\infty$. If $\Lambda$ satisfy both the Conditions \textup{(HNE)} and \textup{(HEDL)}, then the Kumjian--Pask algebra $\KP_\mathsf{F}(\Lambda)$ is a crossed product. 
\end{thm}
\begin{proof}
By Proposition \ref{pro TM distinguishing crosspdt KPAs}, it suffices to show that in $T_\Lambda$, $^\N\epsilon_\Lambda=\epsilon_\Lambda$ for all $\N\in \mathbb{Z}^k$. It is enough if we can show this for $\N=\mathbf{e}_i$, $i=1,2,\ldots,k$. So choose any $i\in \{1,2,\ldots,k\}$. Since $\Lambda$ has Condition (HEDL), for any $\mu\in \mathcal{C}_i$, there exists $\ell^\mu\in \mathbb{N}\setminus \{0\}$ such that $|\mathcal{O}_t^\mu|=\ell^\mu$ for all $t=0,1,2,\ldots,|\mu|-1$. Now using Lemma \ref{lem vertex representation as sum cycle elements}, we can write 
\begin{align*}
\epsilon_\Lambda&=\displaystyle{\sum_{v\in \Lambda^0}}\left(\displaystyle{\sum_{\mu\in \mathcal{C}_i}}\left(\displaystyle{\sum_{t=0}^{|\mu|-1}}j_t^{\mu,v}s(\mu)(t\mathbf{e}_i)\right)\right)=\displaystyle{\sum_{\mu\in \mathcal{C}_i}}\left(\displaystyle{\sum_{t=0}^{|\mu|-1}}\left(\displaystyle{\sum_{v\in \Lambda^0}}j_t^{\mu,v}\right)s(\mu)(t\mathbf{e}_i)\right)\\
&= \displaystyle{\sum_{\mu\in \mathcal{C}_i}} \left(\displaystyle{\sum_{t=0}^{|\mu|-1}} \ell^\mu s(\mu)(t\mathbf{e}_i)\right)=\displaystyle{\sum_{\mu\in \mathcal{C}_i}} \left(\displaystyle{\sum_{t=0}^{|\mu|-1}} \ell^\mu s(\mu)((t+1)\mathbf{e}_i)\right)\\
&=\displaystyle{\sum_{\mu\in \mathcal{C}_i}}\left(\displaystyle{\sum_{t=0}^{|\mu|-1}}\left(\displaystyle{\sum_{v\in \Lambda^0}}j_t^{\mu,v}\right)s(\mu)(t\mathbf{e}_i+\mathbf{e}_i)\right)=\displaystyle{\sum_{v\in \Lambda^0}}\left(\displaystyle{\sum_{\mu\in \mathcal{C}_i}}\left(\displaystyle{\sum_{t=0}^{|\mu|-1}}j_t^{\mu,v}s(\mu)(t\mathbf{e}_i+\mathbf{e}_i)\right)\right)=\displaystyle{\sum_{v\in \Lambda^0}} v(\mathbf{e}_i)=~^{\mathbf{e}_i}\epsilon_\Lambda.
\end{align*}
where the third equality follows since $\displaystyle{\sum_{v\in \Lambda^0}}j_t^{\mu,v}=\displaystyle{\sum_{v\in \Lambda^0}}|v\Lambda\cap \mathcal{O}_t^\mu|=|\mathcal{O}_t^\mu|=\ell^\mu$, and the fourth equality follows since $s(\mu)=s(\mu)(d(\mu))$. This completes the proof. 
\end{proof}

\begin{rmk}\label{rem an alternative proof}
We present an alternative clever proof of Theorem \ref{th sufficient condition for crosspdt}. Suppose we have all the conditions as described in Theorem \ref{th sufficient condition for crosspdt}. For each $j=1,2,\ldots,k$, let $\iota_j:\mathbb{N}\longrightarrow \mathbb{N}^k$ be the injective homomorphism defined by $\iota_j(x):=x\mathbf{e}_j$ for all $x\in \mathbb{N}$. Then it is easy to observe that each of the $k$ coordinate graphs $\Lambda_j:=\iota_j^*(\Lambda)$, $j=1,2,\ldots,k$ has Condition (HEDL) as $1$-graphs. Also, $\Lambda_j$ has Condition (HNE) by Remark \ref{rem proof holds for any f}. Therefore the transpose graphs $\Lambda_j^{T}$ have Conditions (NE) and (EDL) of \cite{Vas}, implying that the Kumjian--Pask algebras $\KP_\mathsf{F}(\Lambda_j)$ are crossed products since $\KP_\mathsf{F}(\Lambda_j)\cong L_\mathsf{F}(\Lambda_j^{T})$. Then by Proposition \ref{pro TM distinguishing crosspdt KPAs} $^1\epsilon_{\Lambda_j}=\epsilon_{\Lambda_j}$ in $T_{\Lambda_j}$ for all $j=1,2,\ldots,k$. Finally applying Proposition \ref{pro TM under pullbacks} and noting that the homomorphism $\rho$ preserves order-units, we have $^{\mathbf{e}_j}\epsilon_\Lambda=\epsilon_\Lambda$ for all $j=1,2,\ldots,k$. Therefore, $^\N \epsilon_\Lambda=\epsilon_\Lambda$ for all $\N\in \mathbb{N}^k$ and hence, $\KP_\mathsf{F}(\Lambda)$ is a crossed product by Proposition \ref{pro TM distinguishing crosspdt KPAs}.\qed
\end{rmk}

We end this section with two examples. The first one shows that we cannot simply drop the Condition (HEDL) in Theorem \ref{th sufficient condition for crosspdt}. 

\begin{example}\label{ex HEDL cannot be removed}
Let $\Lambda$ be the $2$-graph with the following $1$-skeleton.
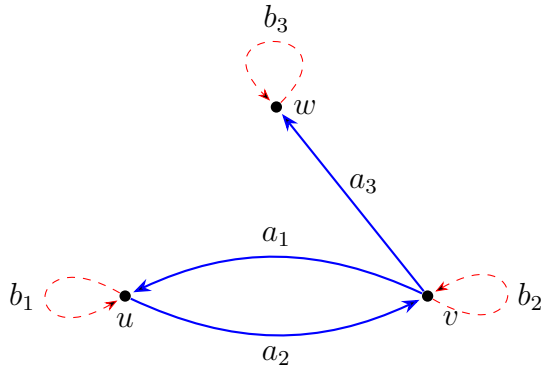
\begin{center}
\begin{tikzpicture}[
    >=Stealth,
    node distance=3cm,
    vertex/.style={circle, fill=black, inner sep=1.5pt},
    blue edge/.style={draw=blue, thick, ->, shorten >=1pt},
    red loop/.style={draw=red, dashed, ->, shorten >=1pt}
]

    \node[vertex, label=below:$u$] (u) at (0,0) {};
    \node[vertex, label=below right:$v$] (v) at (4,0) {};
    \node[vertex, label=right:$w$] (w) at (2,2.5) {};

    \path[blue edge] (v) edge[bend right=25] node[above, black] {$a_1$} (u);
    \path[blue edge] (u) edge[bend right=25] node[below, black] {$a_2$} (v);
    \path[blue edge] (v) edge node[right, black, pos=0.6] {$a_3$} (w);

    \path[red loop] (u) edge[loop left, out=150, in=210, looseness=8, distance=1.5cm] node[left, black] {$b_1$} (u);
    \path[red loop] (v) edge[loop right, out=330, in=30, looseness=8, distance=1.5cm] node[right, black] {$b_2$} (v);
    \path[red loop] (w) edge[out=45, in=135, looseness=8, distance=1.5cm] node[above, black] {$b_3$} (w);

\end{tikzpicture}
\end{center}
As usual, blue solid edges have degree $\mathbf{e}_1=(1,0)$ and red dashed edges have degree $\mathbf{e}_2=(0,1)$. Clearly, $\Lambda$ has Condition (HNE). There are three unicolor cycles in degree $\mathbf{e}_2$, namely $\eta_1:=b_1$, $\eta_2:=b_2$ and $\eta_3=b_3$. Note that $|\mathcal{O}_0^{\eta_1}|=|\mathcal{O}_0^{\eta_2}|=|\mathcal{O}_0^{\eta_3}|=1$. So $\Lambda$ satisfies the requirement of Condition (HEDL) for $i=2$. However, it does not satisfy the requirement for $i=1$ as there is a unicolor cycle in degree $\mathbf{e}_1$, namely $\mu:=a_2 a_1$ such that $|\mathcal{O}_0^\mu|=1\neq 2=|\mathcal{O}_1^\mu|$. Hence, $\Lambda$ fails to have Condition (HEDL). In $T_\Lambda$, $u=v(\mathbf{e}_1)$, $v=u(\mathbf{e}_1)$ and $w=v(\mathbf{e}_1)$ but $v\neq v(\mathbf{e}_1)$. These in turn imply $^{\mathbf{e}_1}\epsilon_\Lambda\neq \epsilon_\Lambda$ and so $\KP_\mathsf{F}(\Lambda)$ is not a crossed product by Proposition \ref{pro TM distinguishing crosspdt KPAs}.   
\end{example}

The conditions of Theorem \ref{th sufficient condition for crosspdt} are also necessary for the Leavitt path algebra of a directed graph to be a crossed product (see \cite[Theorem 3.1]{Vas}). In the level of general $k$-graphs, although we have seen that Condition (HNE) is necessary for crossed product Kumjian--Pask algebras (Corollary \ref{cor HNE is necessary to be crosspdt}), the Condition (HEDL) is not necessary. Our next example justifies this. This again showcases the structural complexity of general $k$-graphs compared to directed graphs. 
\begin{example}\label{ex HEDL is not necessary}
Consider a $2$-graph $\Lambda$ with the following $1$-skeleton:
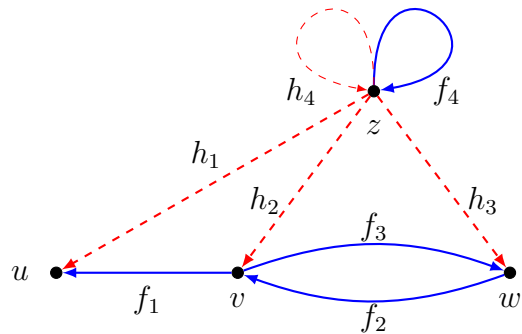
\[
\begin{tikzpicture}[scale=1.2]

\node[inner sep=1.5pt, circle,draw,fill=black] (A) at (0, 0) {};
\node[inner sep=1.5pt, circle,draw,fill=black] (B) at (2,0) {};
\node[inner sep=1.5pt, circle,draw,fill=black] (C) at (5,0) {};
\node[inner sep=1.5pt, circle,draw,fill=black] (D) at (3.5,2) {}
edge[-latex, blue,thick,loop, in=10, out=90, min distance=50] (D)
edge[-latex, red, dashed, loop, out=90, in=170, min distance=50] (D);
	
\path[->, red, dashed, >=latex,thick] (D) edge [] node[above=0.05cm]{} (A);
\path[->, red, dashed, >=latex,thick] (D) edge [] node[above=0.05cm]{} (B);
\path[->, red, dashed, >=latex,thick] (D) edge [] node[above=0.05cm]{} (C);
\path[->, blue, >=latex,thick] (C) edge [bend left=20] node[above=0.05cm]{} (B);
\path[->, blue, >=latex,thick] (B) edge [bend left=20] node[above=0.05cm]{} (C);
\path[->, blue, >=latex,thick] (B) edge [] node[above=0.05cm]{} (A);

\node at (3.5,1.6) {$z$};
\node at (-0.4,0) {$u$};
\node at (2,-0.3) {$v$};
\node at (5,-0.3) {$w$};
\node at (1,-0.3) {$f_1$};
\node at (3.5,-0.5) {$f_2$};
\node at (3.5,0.5) {$f_3$};
\node at (4.3,2) {$f_4$};
\node at (2.7,2) {$h_4$};
\node at (1.65,1.3) {$h_1$};
\node at (2.3,0.8) {$h_2$};
\node at (4.7,0.8) {$h_3$};
\end{tikzpicture}
\]
The factorization rules for bicolored paths are uniquely determined. It is clear that $\Lambda$ is a row-finite $2$-graph without sources. The cycles in the skeleton have no entrances and so $\Lambda$ has Condition (HNE). Consider the cycle $\mu=f_2 f_3$ with source $v$. Then $\mu$ is a unicolor cycle in degree $\mathbf{e}_1$ and $|\mu|=2$. Note that $\mathcal{O}_0^\mu=\{v\}$ and $\mathcal{O}_1^\mu=\{f_1,f_3\}$. Therefore, $|\mathcal{O}_0^\mu|\neq |\mathcal{O}_1^\mu|$ and so $\Lambda$ does not satisfy Condition (HEDL). However, in $T_\Lambda$, $u=v=w=z=z(\mathbf{e}_2)$ and $z=z(\mathbf{e}_1)$. Therefore $^\N z=z$ for all $\N\in \mathbb{Z}^2$ and consequently, $^\N\epsilon_\Lambda=~^\N(u+v+w+z)=~^\N(4z(\mathbf{e}_2))=4(^\N z(\mathbf{e}_2))=4z(\mathbf{e}_2)=\epsilon_\Lambda$ for all $\N\in \mathbb{Z}^2$, which shows that $\KP_\mathsf{F}(\Lambda)$ is a crossed product.
\end{example}

\textbf{Acknowledgments:} The initial discussions on this topic took place in 2025 at the Mathematisches Forschungsinstitut Oberwolfach as part of the Research Pair Program, and the work was further developed during the third author’s visit to Western Sydney University in 2026. Mukherjee wholeheartedly thanks Hazrat for his warm hospitality during the visit. Hazrat acknowledges Australian Research Council Discovery Project DP230103184.

{\bf Competing Interests.} The authors declare that there are no financial or non-financial competing interests related to this work.

\end{document}